\documentclass[a4paper,11pt]{amsart}
\usepackage[colorlinks, linkcolor=blue,anchorcolor=Periwinkle,
    citecolor=blue,urlcolor=Emerald]{hyperref}
\usepackage[all]{xy}
\SelectTips{cm}{}

\usepackage{graphicx}
\usepackage{psfrag}

\usepackage{mathtools}
\usepackage{tikz}
\usepackage{tikz-cd}
\usepackage{adjustbox}
\usepackage{extarrows}
\usepackage{multirow}

\usepackage{amssymb}
\usetikzlibrary{decorations.pathreplacing}
\usetikzlibrary{matrix,arrows}

\usepackage{enumitem}

\usepackage{stmaryrd}

\newcommand{\cf}{{cf.}\ }
\newcommand{\eg}{{e.g.}\ }
\newcommand{\ko}{\: , \;}
\newcommand{\ul}[1]{\underline{#1}}

\numberwithin{equation}{subsection}

\newtheorem{classification-theorem}[subsection]{Classification Theorem}
\newtheorem{decomposition-theorem}[subsection]{Decomposition Theorem}
\newtheorem{proposition-definition}[subsection]{Proposition-Definition}
\newtheorem{periodicity-conjecture}[subsection]{Periodicity Conjecture}

\newtheorem{theorem}{Theorem}
\numberwithin{theorem}{section}
\newtheorem{thmx}{Theorem}

\newtheorem{lemma}[theorem]{Lemma}
\newtheorem{proposition}[theorem]{Proposition}
\newtheorem{propx}[thmx]{Proposition}
\newtheorem{corollary}[theorem]{Corollary}
\theoremstyle{definition}

\theoremstyle{plain}

\newtheorem{question}[theorem]{Question}
\theoremstyle{definition}
\newtheorem{example}[theorem]{Example}
\theoremstyle{plain}

\newtheorem{notation}[theorem]{Notation}

\newcommand{\reminder}[1]{}

\newcommand{\CM}{\mathrm{CM}\,}

\newcommand{\Mod}{\mathrm{Mod}\,}

\newcommand{\per}{\mathrm{per}\,}

\newcommand{\op}{^{op}}

\newcommand{\diag}{\mathrm{diag}}

\newcommand{\Z}{\mathbb{Z}}
\newcommand{\N}{\mathbb{N}}

\newcommand{\iso}{\xrightarrow{_\sim}}
\newcommand{\liso}{\xleftarrow{_\sim}}

\newcommand{\V}{\mathbb{V}}

\newcommand{\Hom}{\mathrm{Hom}}

\newcommand{\Ext}{\mathrm{Ext}}

\newcommand{\End}{\mathrm{End}}

\newcommand{\ten}{\otimes}

\newcommand{\Ga}{\Gamma}

\newcommand{\Si}{\Sigma}

\renewcommand{\phi}{\varphi}

\begin{document}

\date{\today}

\title[Gorenstein tiled orders and incidence algebras of posets]{Triangle equivalences between Gorenstein\\[0.15cm] tiled orders and incidence algebras of posets}

\author{Osamu Iyama}
\address{Graduate School of Mathematical Sciences, The University of Tokyo, 3-8-1 Komaba \linebreak Meguro-ku Tokyo 153-8914, Japan}
\email{iyama@ms.u-tokyo.ac.jp}
\urladdr{https://www.ms.u-tokyo.ac.jp/~iyama/}

\author{Junyang Liu}
\address{School of Mathematical Sciences, University of Science and Technology of China, Hefei 230026, China}
\address{Graduate School of Mathematical Sciences, The University of Tokyo, 3-8-1 Komaba \linebreak Meguro-ku Tokyo 153-8914, Japan}
\email{liuj@imj-prg.fr}
\email{liu@ms.u-tokyo.ac.jp}
\urladdr{https://webusers.imj-prg.fr/~junyang.liu}

\begin{abstract}
We prove that for any $\N$-graded Gorenstein tiled order $A$, the stable category $\ul{\CM \!}^{\Z} A$ is triangle equivalent to the perfect derived category of the incidence algebra of a finite poset $\V_A \op$. Moreover, for a finite poset $P$, we prove that the incidence algebra of $P$ can be realized as the endomorphism algebra of a standard tilting object if and only if $P$ is either empty or has the maximum. We also study the behaviors of the corresponding poset under graded Morita equivalences and coverings of a Gorenstein tiled order. Finally, we classify Gorenstein tiled orders $A$ satisfying $|\V_A \op|\leq 3$.
\end{abstract}




\maketitle

\vspace*{-1cm}
\tableofcontents

\section{Introduction}

The study of (maximal) Cohen--Macaulay modules is one of the central subjects in commutative algebra and representation theory, \cf \cite{Auslander78, CurtisReiner81, Yoshino90, Simson92, LeuschkeWiegand12}. It has a strong connection with algebraic geometry, singularity theory, mathematical physics, and a number of important results are obtained by applying tilting theory for singularity categories associated with commutative and non-commutative Gorenstein rings, \cf~\cite{KajiuraSaitoTakahashi09, FutakiUeda11, KussinLenzingMeltzer13, HerschendIyamaMinamotoOppermann23, MoriUeyama16, LuZhu21, BuchweitzIyamaYamaura20, Iyama18, HiranoOuchi23, HaniharaIyama22, KimuraMinamotoYamaura25, IyamaKimuraUeyama24}. For a commutative Gorenstein ring $A$, the category $\CM A$ of Cohen--Macaulay modules is a Frobenius category. Its stable category is triangle equivalent to the singularity category associated with $A$ introduced by Buchweitz \cite{Buchweitz21} and Orlov \cite{Orlov04}.

Orders are one of the main objects studied in the theory of Cohen--Macaulay representations. Tiled orders are a class of orders whose representation theory as well as homological algebra has been well developed \cite{ZavadskiiKiricenko77, Simson92, KirichenkoKhibinaMashchenkoPlakhotnykZhuravlev14}. Denote $R=k[x]$ and $K=k[x, x^{-1}]$. For a finite set $\mathbb{I}$, a {\it tiled order} (over $R$) is an $R$-subalgebra of the form
\[
A=[Rx^{d(i, j)}] \subseteq {\rm M}_{\mathbb{I}}(K) \: ,
\]
where $d\colon \mathbb{I}\times \mathbb{I}\to \Z$ is a {\it quasi-metric} in the sense that it satisfies $d(i, i)=0$ and \linebreak $d(i, k)+d(k, j)\geq d(i, j)$ for all $i$, $j$, $k\in \mathbb{I}$. We regard $A$ as a $\Z$-graded algebra whose grading is determined by $\deg x=1$.
A basic tiled order $A$ is Gorenstein if and only if there exists a bijection $\nu \colon \mathbb{I} \to \mathbb{I}$ satisfying $d(\nu i, j)+d(j, i)=d(\nu i, i)$ for all $i$, $j\in \mathbb{I}$.

It is known \cite{ZavadskiiKiricenko77, Simson92} that Cohen--Macaulay representations of a tiled order $A$ are closely related to representations of a poset. More explicitly, let $\mathbb{P}_A$ be the poset of projective graded \linebreak $A$-submodules of $\begin{bmatrix}
K & \cdots & K
\end{bmatrix}$. Then the category $\CM\!^{\Z} A$ of Cohen--Macaulay $\Z$-graded \linebreak $A$-modules is equivalent to the category of finitely generated monomorphism representations of $\mathbb{P}_A$. In particular, the tiled order $A$ is of finite Cohen--Macaulay representation type if and only if $\mathbb{P}_A$ contains no critical posets whose Hasse quivers are given as follows, \cf \cite{Kleiner72}.
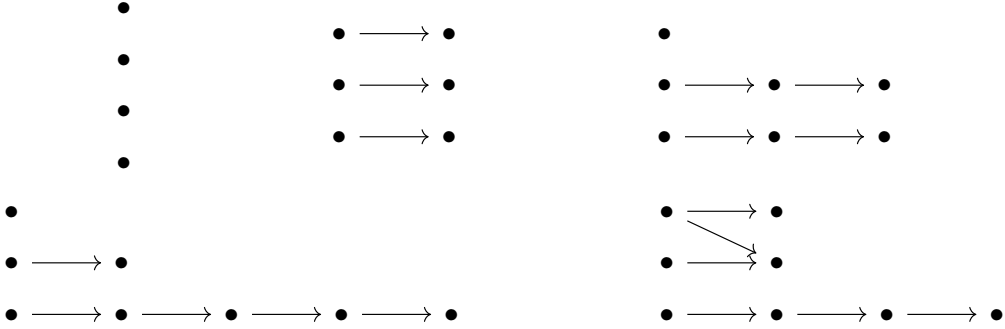
\[
\begin{tikzcd}[row sep=0.6em]
\bullet \\
\bullet \\
\bullet \\
\bullet
\end{tikzcd}
\qquad \qquad \qquad
\begin{tikzcd}[row sep=0.6em]
\bullet \arrow{r} & \bullet \\
\bullet \arrow{r} & \bullet \\
\bullet \arrow{r} & \bullet
\end{tikzcd}
\qquad \qquad \qquad
\begin{tikzcd}[row sep=0.6em]
\bullet & & \\
\bullet \arrow{r} & \bullet \arrow{r} & \bullet \\
\bullet \arrow{r} & \bullet \arrow{r} & \bullet
\end{tikzcd}
\]
\[
\begin{tikzcd}[row sep=0.6em]
\bullet & & & & \\
\bullet \arrow{r} & \bullet & & & \\
\bullet \arrow{r} & \bullet \arrow{r} & \bullet \arrow{r} & \bullet \arrow{r} & \bullet
\end{tikzcd}
\qquad \qquad \qquad
\begin{tikzcd}[row sep=0.6em]
\bullet \arrow{r} \arrow{dr} & \bullet & & \\
\bullet \arrow{r} & \bullet & & \\
\bullet \arrow{r} & \bullet \arrow{r} & \bullet \arrow{r} & \bullet
\end{tikzcd}
\]

The first aim of this article is to establish another strong connection between Cohen--Macaulay $\Z$-graded $A$-modules and representations of a poset by completing a result in \cite{IyamaKimuraUeyama24}. More explicitly, for an $\N$-graded Gorenstein tiled order $A$, the stable category $\ul{\CM \!}^\Z A$ has a standard tilting object $V=V_A$ whose endomorphism algebra $\Ga=\Ga_A$ is isomorphic to the incidence algebra of a certain finite poset $\V_A\op$ explicitly constructed.



\begin{thmx}[=Theorem~\ref{thm:tilting object}] \label{thm:A}
Let $A$ be an $\N$-graded Gorenstein tiled order. 
\begin{itemize}
\item[a)] The object $V$ in $\ul{\CM \!}^{\Z} A$ is tilting.
\item[b)] The endomorphism algebra of the standard tilting object $V$ in $\ul{\CM \!}^{\Z} A$ is isomorphic to $k\V_A\op$.
\item[c)] The poset $\V_A$ is either empty or a finite poset with minimum.
\end{itemize}
\end{thmx}

In particular, it provides large classes of stable equivalent Gorenstein tiled orders via posets. To study such posets, it is natural to ask the following questions.

\begin{question} \label{que:poset}
Let $P$ be a finite poset.
\begin{itemize}
\item[a)] When can $kP$ be realized as the algebra $\Ga_A$ for a basic $\N$-graded Gorenstein tiled order $A$?
\item[b)] Give a classification of the basic $\N$-graded Gorenstein tiled orders $A$ for which the algebra $\Ga_A$ is isomorphic to the incidence algebra $kP$?
\item[c)] Give a classification of the basic $\N$-graded Gorenstein tiled orders $A$ for which the stable category $\ul{\CM \!}^{\Z} A$ is triangle equivalent to $\per kP$?
\end{itemize}
\end{question}

To answer part~a) of Question~\ref{que:poset}, we give the following simple criterion.

\begin{thmx}[see Theorem~\ref{thm:realization of posets} for details] \label{thm:B}
Let $P$ be a finite poset. Then there exists an $\N$-graded Gorenstein tiled order $A$ such that the algebra $\Ga_A$ is isomorphic to the incidence algebra $kP$ if and only if $P$ is either the empty poset or has the maximum.
\end{thmx}

A finite poset $\mathbb{I}$ gives rise to a quasi-metric $d$ on $\mathbb{I}$ given by $d(i, j)=0$ if $i \leq j$, otherwise $d(i, j)=1$. Our construction of $A$ in Theorem~\ref{thm:B} is given by the $\N$-graded Gorenstein tiled order determined by the quasi-metric space
\[
(\mathbb{I}\sqcup \mathbb{I}, \begin{bmatrix}
d & \mathbf{1}-d^t \\
\mathbf{2}-d^t & d
\end{bmatrix})\: ,
\]
where we denote $P\setminus \{0\}$ by $\mathbb{I}$ and the matrix with all entries $1$ by $\mathbf{1}$.

Since the object $V$ is basic, Proposition~\ref{prop:D} below shows that the poset $\V_A$ does not change under taking coverings of $A$. To study part~b) of Question~\ref{que:poset}, it suffices to classify $A$ up to coverings. Proposition~\ref{prop:conjugacy} shows that the poset $\V_A$ changes a little under taking conjugacies of $A$. To study part~c) of Question~\ref{que:poset}, it suffices to classify $A$ up to conjugacies. We give the answer to both parts~b) and c) of Question~\ref{que:poset} for suitable finite posets $P$. Notice that the answer to part~c) can be deduced from that to part~b).

\begin{thmx}[see Theorems~\ref{thm:classification 1} and \ref{thm:classification 2} for details] \label{thm:C}
The basic $\N$-graded Gorenstein tiled order $A$ such that the algebra $\Ga_A$ vanishes or is isomorphic to $kA_1$, $kA_2$, $kA_3$, respectively, admits a complete classification.
\end{thmx}

In Proposition~\ref{prop:global dimension 2}, we prove that $\Ga_A$ has global dimension at most $2$ for a special class of cyclic Gorenstein orders. Lemma~\ref{lem:rank 3} shows that Gorenstein tiled orders given by quasi-metric space of cardinality $3$ must be of this form. In this case, we classify the Gorenstein tiled orders such that $\Ga_A$ are hereditary, \cf Proposition~\ref{prop:hereditary of size 3}.

To prove our results, we give some preliminary results on graded Morita equivalences and covering theory on $\Z$-graded algebras, including Artin--Schelter Gorenstein algebras. In Section~\ref{ss:graded Morita equivalences}, we recall graded Morita equivalences. In Section~\ref{ss:coverings of graded algebras}, we study coverings of \linebreak $\Z$-graded rings. They establish a connection between numerical semigroup algebras and tiled orders, \cf Example~\ref{ex:covering of numerical semigroup algebra}. We give an explicit formula for the exponent matrices associated with coverings of a tiled order, \cf Example~\ref{ex:covering of tiled order}. In Proposition~\ref{prop:covering of Gorenstein order}, we prove that a covering of an Artin--Schelter Gorenstein algebra is also an Artin--Schelter Gorenstein algebra and give explicit formulas of its Nakayama permutation and Gorenstein parameters. For a $\Z$-graded ring $A$, it is well-known \cite{Gabriel81, DaoIyamaTakahashiWemyss20} that we have an equivalence
\[
F\colon \Mod \! ^{\Z} A \xlongrightarrow{_\sim} \Mod \! ^{\Z}A^{[m]}\: .
\]
If $A$ is an Iwanaga--Gorenstein $\Z$-graded algebra, then the category $\CM\!^{\Z} A$ has a Frobenius subcategory $\CM\!^{\Z}_0 A$ which enjoys Auslander--Reiten--Serre duality \cite{Auslander78}. The equivalence $F$ restricts to $\CM \!^{\Z}_0 A \iso \CM \!^{\Z}_0 A^{[m]}$ and induces a stable equivalence
\[
\ul{\CM \!}^{\Z}_0 A \xlongrightarrow{_\sim} \ul{\CM \!}^{\Z}_0 A^{[m]}\: .
\]
More explicitly, we prove that this stable equivalence preserves the standard silting objects up to direct summands.

\begin{propx}[=Proposition~\ref{prop:standard silting}] \label{prop:D}
Let $A$ be a ring-indecomposable $\N$-graded Artin--Schelter Gorenstein algebra of dimension $1$ such that the global dimension of $A_0$ is finite. Denote the standard silting object in $\ul{\CM \!}^{\Z}_0 A$ (respectively, $\ul{\CM \!}^{\Z}_0 A^{[m]}$) by $V_A$ (respectively, $V_{A^{[m]}}$).
\begin{itemize}
\item[a)] The additive subcategories of $\ul{\CM \!}^{\Z}_0 A^{[m]}$ generated by $FV_A$ and $V_{A^{[m]}}$ coincide.
\item[b)] The endomophism algebras of the standard silting objects in $\ul{\CM \!}^{\Z}_0 A$ and $\ul{\CM \!}^{\Z}_0 A^{[m]}$ are Morita equivalent.
\end{itemize}
\end{propx}

In Proposition~\ref{prop:compatibility}, we show that graded Morita equivalences and coverings are compatible. We end this section by giving the following notation.

\begin{notation}
The following notation is used throughout the article: We let $k$ be a field. Algebras have units and morphisms of algebras preserve the units. Modules are unital right modules. For a $\Z$-graded ring $A$, we denote the category of $\Z$-graded $A$-modules by $\Mod \! ^{\Z} A$. We denote the shift functor of $\Z$-graded vector spaces by $(1)$ so that we have $V(1)_n =V_{n+1}$. We denote the truncation of a $\Z$-graded vector space $V$ in non-negative degrees by $V_{\geq 0}$. For a $\Z$-graded algebra $A$, we fix a complete set $\mathbb{I}$ of representatives of primitive orthogonal idempotents of $A$ modulo the equivalence relation $\sim$, where $i\sim j$ if and only if $e_i A$ and $e_j A$ are isomorphic as ungraded $A$-modules.
\end{notation}

\subsection*{Acknowledgments}
The first-named author is supported by JSPS Grant-in-Aid for Scientific Research (B) 23K22384. The second-named author is supported by Grant-in-Aid for JSPS Fellows 25KF0130.

\section{Preliminaries}

\subsection{Artin--Schelter Gorenstein algebras} \label{ss:AS-Gorenstein algebras}

Let $A=\bigoplus_{j\in \Z}A_j$ be a basic noetherian $\Z$-graded algebra $A$ such that $A_j$ is finite-dimensional for all $j\in \Z$ and $A_j$ vanishes for all $j\ll 0$. For any $i\in \mathbb{I}$, we denote the top of the graded $A$-module $e_i A$ (respectively, $A\op$-module $Ae_i$) by $S_i$ (respectively, $S_i \op$). The $\Z$-graded algebra $A$ is called {\it Artin--Schelter Gorenstein of dimension $d$} if
\begin{itemize}
\item[a)] it has injective dimension $d$ as both a left and a right module over itself,
\item[b)] there exist a bijection $\nu \colon \mathbb{I} \to \mathbb{I}$ and integers $p_i$ such that we have an isomorphism $\Ext_A^l(S_i, A)\simeq S_{\nu i}\op(p_i)$ of graded $A\op$-modules if $l=d$, and $\Ext_A^l(S_i, A)$ vanishes if $l\neq d$, $i\in \mathbb{I}$.
\end{itemize}
The bijection $\nu$ is called the {\it Nakayama permutation} and the integers $p_i$ are called the {\it Gorenstein parameters} of $A$. These are fundamental invariants of Artin--Schelter Gorenstein algebras.

Let $A$ be an Artin--Schelter Gorenstein algebra. A finitely generated graded $A$-module $M$ is {\it (maximal) Cohen--Macaulay} if $\Ext_A ^i(M, A)$ vanishes for all positive integers $i$. We write $\CM \!^{\Z} A$ for the category whose objects are $\Z$-graded $A$-modules which are Cohen--Macaulay and morphisms are $A$-homomorphisms which are homogeneous of degree $0$. Denote the graded total quotient algebra of $A$ by $Q$. The category $\CM \!_0^\Z A$ is defined to be the full subcategory of $\CM \!^{\Z} A$ whose objects are the $\Z$-graded $A$-modules $M$ such that the \linebreak $\Z$-graded $Q$-module $M\ten_A Q$ is projective. Both categories $\CM \!^{\Z} A$ and $\CM \!_0^\Z A$ are Frobenius. They coincide if and only if $A$ has at worst isolated singularities.

\subsection{Gorenstein orders} \label{ss:Gorenstein orders}

From now on, we always denote $R=k[x]$ and $K=k[x, x^{-1}]$ with $\deg x=1$.
A module-finite $\N$-graded $R$-algebra $A$ is called an {\it $R$-order} if it is projective as an $R$-module.
In this case, the graded $A$-bimodule $\omega=\Hom_R (A, R(-1))$ is called a {\it graded canonical module} of $A$.
An $R$-order $A$ is called a {\it Gorenstein $R$-order} if $\omega$ is a projective graded $A$-module. Basic Gorenstein $R$-orders are Artin--Schelter Gorenstein algebras. The Nakayama permutation and the Gorenstein parameters of a basic Gorenstein $R$-order $A$ are determined by the isomorphisms $e_i \omega \simeq e_{\nu i}A(-p_i)$ of graded $A$-modules, $i\in \mathbb{I}$. If the action of $\nu$ on $\mathbb{I}$ is transitive, then $A$ is called a {\it cyclic Gorenstein order}.

Let $A$ be a Gorenstein $R$-order. A finitely generated graded $A$-module $M$ is Cohen--Macaulay if and only if it is projective as an $R$-module. An object $M\in \CM \!^{\Z} A$ lies in $\CM \!_0^\Z A$ if and only if the $\Z$-graded $A\ten_R K$-module $M\ten_R K$ is projective.

\subsection{Gorenstein tiled orders} \label{ss:Gorenstein tiled orders}

Following Section~8 of \cite{IyamaKimuraUeyama24}, a {\it tiled order} (over $R$) is an $R$-subalgebra
\[
A=
\begin{bmatrix}
Rx^{d(1, 1)} & \cdots & Rx^{d(1, n)} \\
\vdots & \ddots & \vdots \\
Rx^{d(n, 1)} & \cdots & Rx^{d(n, n)}
\end{bmatrix}
\subseteq {\rm M}_n(K)
\]
for a positive integer $n$, where $d(i, j)$ are integers for all $1\leq i$, $j\leq n$. The algebra ${\rm M}_n(K)$ has a $\Z$-grading given by ${\rm M}_n(K)_j={\rm M}_n(k)x^j$. Clearly, the algebra $A$ inherits this \linebreak $\Z$-grading. Notice that $A$ is an $R$-subalgebra of ${\rm M}_n(K)$ if and only if we have $d(i, i)=0$ and $d(i, k)+d(k, j)\geq d(i, j)$ for all $1\leq i$, $j$, $k\leq n$, that is, the map $d$ is a quasi-metric on $\{1, \ldots, n\}$. In this case, we define the {\it exponent matrix} of $A$ as
\[
{\rm v}(A)=[d(i, j)]\in {\rm M}_n(\Z) \: .
\]
It is basic if and only if we have $d(i, j)+d(j, i)>0$ for all $i\neq j$. A basic tiled order $A$ is a Gorenstein order if and only if there exists a bijection $\nu \colon \mathbb{I} \to \mathbb{I}$ such that we have $d(\nu i, j)+d(j, i)=d(\nu i, i)$ for all $i$, $j\in \mathbb{I}$. In this case, the bijection $\nu$ is the Nakayama permutation and the integers $p_i=1-d(\nu i, i)$ are the Gorenstein parameters of $A$. For convenience, we call the entries $d(\nu i, i)=1-p_i$ the {\it $b$-invariants} of $A$, following a custom that $-p_i$ are called the $a$-invariants of $A$. Notice that if moreover $A$ is $\N$-graded, then an entry being a $b$-invariant implies that it is the greatest entry in the same row and column of ${\rm v}(A)$.

Let $A$ be an $\N$-graded Gorenstein tiled order. By replacing it with a graded Morita equivalent basic $\N$-graded Gorenstein tiled order, from part~(1) of Theorem~5.1 of \cite{IyamaKimuraUeyama24}, we deduce that the stable category $\ul{\CM \!}^\Z A$ contains a standard silting object
\[
V=V_A=\begin{bmatrix}
R & \cdots & R
\end{bmatrix}\oplus \bigoplus_{i\in \mathbb{I}}\bigoplus_{j=1}^{-p_i}e_{\nu i}A(j)_{\geq 0} \: .
\]
Notice that our $V$ is basic, but the one in \cite{IyamaKimuraUeyama24} is not basic in general. For any Cohen--Macaulay graded $A$-module
\[
M={\rm L}(v)=
\begin{bmatrix}
Rx^{v_1} & \cdots & Rx^{v_n}
\end{bmatrix}
\]
with integers $v_1$, \ldots, $v_n$, we denote ${\rm v}(M)=(v_1, \ldots, v_n)$. We define the subposet
\[
\V_A=\{{\rm v}(e_i A(j)_{\geq 0})\mid i\in \mathbb{I}, j\geq 1\}\setminus \{{\rm v}(e_i A)\mid i\in \mathbb{I}\} \subseteq \Z^n \: .
\]
If all $b$-invariants of $A$ are positive, then we have
\[
\V_A=\{{\rm v}(e_i A(j)_{\geq 0})\mid i\in \mathbb{I}, j\geq 1\}\: .
\]
In this case, the standard silting object $V$ is tilting and the endomorphism algebra \linebreak $\Ga=\Ga_A=\End_{\ul{\CM \!}^\Z A}(V)$ is isomorphic to the incidence algebra $k\V_A \op$, \cf Theorem~8.2 of \cite{IyamaKimuraUeyama24}. If one of the $b$-invariants of $A$ is zero, then we cannot apply this theorem. But in Theorem~\ref{thm:tilting object} below, we show it is still true. The following example serves to illustrate the construction of $\V_A \op$ in this case.

\begin{example} \label{ex:poset}
Let $m$ and $n$ be positive integers. Let $A$ be the tiled order determined by the exponent matrix
\[
\begin{bsmallmatrix}
0 & 0 & \cdots & 0 & 0 \\
m & 0 & \cdots & 0 & 0 \\[-5pt]
\vdots & \vdots & \ddots & \vdots & \vdots \\
m & m & \cdots & 0 & 0 \\
m & m & \cdots & m & 0
\end{bsmallmatrix}
\]
of size $n$. Then the Hasse quiver of $\V_A \op$ is given as follows.
\[
\begin{tikzcd}[ampersand replacement=\&]
\begin{psmallmatrix}
m-1 & 0 & \cdots & 0 & 0
\end{psmallmatrix} \arrow{r}
\& \cdots \arrow{r} \&
\begin{psmallmatrix}
2 & 0 & \cdots & 0 & 0
\end{psmallmatrix} \arrow{r}
\&
\begin{psmallmatrix}
1 & 0 & \cdots & 0 & 0
\end{psmallmatrix} \\
\vdots \arrow{r} \arrow{u} \& \ddots \arrow{r} \arrow{u} \& \vdots \arrow{r} \arrow{u} \& \vdots \arrow{u} \\
\begin{psmallmatrix}
m-1 & m-1 & \cdots & 0 & 0
\end{psmallmatrix} \arrow{r} \arrow{u}
\& \cdots \arrow{r} \arrow{u} \&
\begin{psmallmatrix}
2 & 2 & \cdots & 0 & 0
\end{psmallmatrix} \arrow{r} \arrow{u}
\&
\begin{psmallmatrix}
1 & 1 & \cdots & 0 & 0
\end{psmallmatrix} \arrow{u}\\
\begin{psmallmatrix}
m-1 & m-1 & \cdots & m-1 & 0
\end{psmallmatrix} \arrow{r} \arrow{u}
\& \cdots \arrow{r} \arrow{u} \&
\begin{psmallmatrix}
2 & 2 & \cdots & 2 & 0
\end{psmallmatrix} \arrow{r} \arrow{u}
\&
\begin{psmallmatrix}
1 & 1 & \cdots & 1 & 0
\end{psmallmatrix} \arrow{u}
\end{tikzcd}
\]
\end{example}

\section{On the virtues of graded algebras}

To study representations of $\Z$-graded algebras, the notions of graded Morita equivalences and coverings are fundamental. The aim of this section is to study behaviors of Nakayama permutations, Gorenstein parameters, and standard silting objects under these two operations.

\subsection{Graded Morita equivalences} \label{ss:graded Morita equivalences}

Let $A$ be a $\Z$-graded algebra and $\{e_1, \ldots, e_n\}$ a complete set of primitive orthogonal idempotents of $A_0$. For integers $l_1$, \ldots, $l_n$, we define a graded $A$-module $P=\bigoplus_{i=1}^n e_i A(l_i)$. Let $B$ be its graded endomorphism algebra. As an ungraded algebra, it is isomorphic to $A$.

\begin{proposition}[\cite{GordonGreen82}]
Let $A$ and $B$ be as above. Then there exists an equivalence $F\colon \Mod \! ^{\Z} A \iso \Mod \! ^{\Z} B$ which is compatible with the shift functors $(1)$ on both sides.
\end{proposition}

The $\Z$-graded algebra $B$ is isomorphic to $(\sum_{i=1}^n e_i(l_i))A(\sum_{i=1}^n e_i(-l_i))$. Any graded Morita equivalence is a composition of those induced by the graded $A$-module $P$ with exactly one nonzero $l_i$ which equals $1$ or $-1$. When $A$ is a $\Z$-graded Gorenstein tiled order, the above construction of $B$ is given by its conjugacy $DAD^{-1}$, where $D$ is the diagonal matrix $\diag(x^{-l_1}, \ldots, x^{-l_n})$. If $A$ and $B$ are $\N$-graded, then the endomorphism algebras of the standard tilting objects in $\ul{\CM \!}^{\Z} A$ and $\ul{\CM \!}^{\Z} B$ are only derived equivalent but not Morita equivalent in general. So the posets $\V_A$ and $\V_B$ may be not isomorphic.

\begin{proposition} \label{prop:conjugacy}
Let $A$ be an $\N$-graded Gorenstein tiled order and $l_1$, \ldots, $l_n$ integers. Denote $D=\diag(x^{l_1}, \ldots, x^{l_n})$. Suppose that $DAD^{-1}$ is also $\N$-graded.
\begin{itemize}
\item[a)] If we have $l_i=1$ and $l_j=0$ for all $j\neq i$, then the poset $\V_{DAD^{-1}}$ is isomorphic to
\[
(\V_A \setminus \{{\rm v}(e_{\nu i}A(1)_{\geq 0})\})\cup \{{\rm v}(e_i A)\}\: .
\]
\item[b)] If we have $l_i=-1$ and $l_j=0$ for all $j\neq i$, then the poset $\V_{DAD^{-1}}$ is isomorphic to
\[
(\V_A \setminus \{{\rm v}(e_i A(1)_{\geq 0})\})\cup \{{\rm v}(e_{\nu i} A)\} \: .
\]
\end{itemize}
\end{proposition}

\begin{proof}
a) Notice that the associated exponent matrix ${\rm v}(DAD^{-1})$ is obtained from ${\rm v}(A)$ by adding $1$ to the $i$-th row and substract $1$ from the $i$-th column. This operation does not change the order between elements of the poset $\V_A$ but removes the maximum ${\rm v}(e_{\nu i}A(1)_{\geq 0})$ associated with the $i$-th column and adds a new maximum ${\rm v}(e_i A)$ associated with the $i$-th row. Then the statement follows.

b) Similar to part~a).
\end{proof}

\subsection{Coverings of graded algebras} \label{ss:coverings of graded algebras}

For a $\Z$-graded ring $A$ and a positive integer $m$, the {\it $m$-th covering} $A^{[m]}$ is defined \cite{Gabriel81, DaoIyamaTakahashiWemyss20} as the $\Z$-graded subring
\[
\begin{bmatrix}
A_{m\Z} & A_{m\Z+1} & \cdots & A_{m\Z+m-1} \\
A_{m\Z-1} & A_{m\Z} & \cdots & A_{m\Z+m-2} \\
\vdots & \vdots & \ddots & \vdots \\
A_{m\Z-m+1} & A_{m\Z-m+2} & \cdots & A_{m\Z}
\end{bmatrix}
\subseteq {\rm M}_m(A) \: .
\]
For example, coverings establish a connection between semigroup algebras and tiled orders.

\begin{example} \label{ex:covering of numerical semigroup algebra}
Let $S\subset \N$ be a cofinite proper submonoid and $A=k[S]\subset k[x]$ the numerical semigroup algebra. We regard $A$ as a $\Z$-graded algebra whose grading is determined by $\deg x=1$. Denote its $a$-invariant by $a$. Then for an integer $m\geq a+1$, the $m$-th covering $A^{[m]}$ is a tiled order.
\end{example}

Coverings of a tiled order are described as follows.

\begin{example} \label{ex:covering of tiled order}
Let $A$ be a tiled order. Then the covering $A^{[m]}$ is isomorphic to the tiled order determined by the exponent matrix
\[
\begin{bmatrix}
\lceil \frac{{\rm v}(A)}{m}\rceil & \lceil \frac{{\rm v}(A)-1}{m}\rceil & \cdots & \lceil \frac{{\rm v}(A)-m+1}{m}\rceil \\
\lceil \frac{{\rm v}(A)+1}{m}\rceil & \lceil \frac{{\rm v}(A)}{m}\rceil & \cdots & \lceil \frac{{\rm v}(A)-m+2}{m}\rceil \\
\vdots & \vdots & \ddots & \vdots \\
\lceil \frac{{\rm v}(A)+m-1}{m}\rceil & \lceil \frac{{\rm v}(A)+m-2}{m}\rceil & \cdots & \lceil \frac{{\rm v}(A)}{m}\rceil
\end{bmatrix} \: .
\]
\end{example}

\begin{proof}
Put $P=
\begin{bsmallmatrix}
I & 0 & \cdots & 0 \\
0 & xI & \cdots & 0 \\[-5pt]
\vdots & \vdots & \ddots & \vdots \\
0 & 0 & \cdots & x^{m-1}I
\end{bsmallmatrix}$, where the size of the identity matrix $I$ is the same as that of ${\rm v}(A)$. After replacing $A^{[m]}$ by $PA^{[m]}P^{-1}$ and $x^m$ by $x$ we obtain the desired tiled order.
\end{proof}

We have the equivalence
\begin{equation} \label{eq:covering}
F\colon \Mod \! ^{\Z} A \iso \Mod \! ^{\Z}A^{[m]}
\end{equation}
which maps $X$ to
$\begin{bmatrix}
X_{m\Z} & X_{m\Z +1} & \cdots & X_{m\Z +m-1}
\end{bmatrix}$.
Similarly, we have the equivalence $F'\colon \Mod \! ^{\Z} A\op \iso \Mod \! ^{\Z}(A^{[m]})\op$ which maps $X$ to
$\begin{bmatrix}
X_{m\Z} & X_{m\Z -1} & \cdots & X_{m\Z -m+1}
\end{bmatrix}^t$.
Denote the idempotent $e_i$ of the $j$-th diagonal block of $A^{[m]}$ by $e_{i, j-1}$. We have a natural bijection $\mathbb{I}_{A^{[m]}}\simeq \mathbb{I}_A \times \{0, \ldots, m-1\}$.
Put $(i, j)=(i, j-\lfloor \frac{j}{m} \rfloor m)$ for all $i\in \mathbb{I}_A$ and integers $j$. We describe the Nakayama permutation and Gorenstein parameters of coverings in the following theorem.

\begin{proposition} \label{prop:covering of Gorenstein order}
Let $A$ be an Artin--Schelter Gorenstein algebra of dimension $d$ with the Nakayama permutation $\nu$ and the Gorenstein parameters $p_i$, $i\in \mathbb{I}$. Then the covering $A^{[m]}$ is an Artin--Schelter Gorenstein algebra of dimension $d$ whose Nakayama permutation is given by $\nu^{[m]}(i, j)=(\nu i, p_i+j)$ and Gorenstein parameters are given by $p^{[m]}_{i, j}=\lfloor \frac{p_i+j}{m}\rfloor$, $(i, j)\in \mathbb{I} \times \{0, \ldots, m-1\}$.
\end{proposition}

\begin{proof}
Since $A$ is an Artin--Schelter Gorenstein algebra of dimension $d$, the injective dimension of the graded $A$-module $\bigoplus_{j=0}^{m-1}A(-j)$ is $d$. So the injective dimension of its image under the equivalence $F$ is also $d$. By the definition of $F$, this image is isomorphic to $A^{[m]}$ as a graded $A^{[m]}$-module. Therefore, the self-injective dimension of $A^{[m]}$ as a right module is $d$. Similarly, the self-injective dimension of $A^{[m]}$ as a left module is $d$.

We now prove the formulas of the Nakayama permutation $\nu^{[m]}$ and the Gorenstein parameters $p^{[m]}_{i, j}$. Since $F$ (respectively, $F'$) induces graded Morita equivalence between $A$ and $A^{[m]}$ (respectively, between $A\op$ and $(A^{[m]})\op$), the composed functor $F'\circ F$ induces an equivalence between the categories of graded bimodules over $A$ and $A^{[m]}$ which maps $A$ to $A^{[m]}$. So we have the isomorphism $\Ext_{A^{[m]}}^l(F(S_i(j)), A^{[m]})\liso F'\Ext_A^l(S_i(j), A)$ of graded $(A^{[m]})\op$-modules for all $i\in \mathbb{I}$, integers $j$, and non-negative integers $l$. If we have $l\neq d$, then $\Ext_{A^{[m]}}^l(F(S_i(j)), A^{[m]})$ vanishes. If we have $l=d$, then we have an isomorphism $\Ext_{A^{[m]}}^l(F(S_i(j)), A^{[m]})\liso F'(S_{\nu i}\op(p_i-j))$ of graded $(A^{[m]})\op$-modules. We conclude the desired formulas from the fact that we have $F(S_i(j))=S_{i, -j}(\lceil \frac{j}{m}\rceil)$ and \linebreak $F'(S_i\op(j))=S_{i, j}\op(\lfloor \frac{j}{m}\rfloor)$ for all $i\in \mathbb{I}$ and integers $j$.
\end{proof}

As a consequence, we obtain the following properties of coverings. They can be used to detect whether an Artin--Schelter Gorenstein algebra is a covering of another one, \eg~the classes in Theorem~\ref{thm:classification 2}.

\begin{proposition}
Let $A$ be an Artin--Schelter Gorenstein algebra.
\begin{itemize}
\item[a)] The number of Gorenstein parameters of $A$ which equal $1$ respectively $-1$ is invariant under taking coverings.
\item[b)] If $(i, j)$ and $(i', j')$ are in the same orbit under the Nakayama permutation of a covering of $A$, then $i$ and $i'$ are in the same orbit under the Nakayama permutation of $A$. In particular, if the action of the Nakayama permutation is transitive after taking a covering, then so is itself.
\end{itemize}
\end{proposition}

\begin{proof}
These are consequences of Proposition~\ref{prop:covering of Gorenstein order}.
\end{proof}

If $A$ is an Artin--Schelter Gorenstein algebra, then the equivalence $F$ restricts to \linebreak $\CM \! ^{\Z}_0 A \iso \CM \! ^{\Z}_0 A^{[m]}$ and induces the equivalence $\ul{\CM \!}^{\Z}_0 A \iso \ul{\CM \!}^{\Z}_0 A^{[m]}$, still denoted by $F$.

Let $A$ be a ring-indecomposable $\N$-graded Artin--Schelter Gorenstein algebra of dimension $1$ such that the global dimension of $A_0$ is finite. Recall that we denote the graded total quotient algebra of $A$ by $Q$. Let $q_A$ be the minimal positive integer such that $\bigoplus_{i=1}^{q_A}Q(i)$ is an additive generator of the category of finitely generated projective graded $Q$-modules. By part~(1) of Theorem~5.1 of \cite{IyamaKimuraUeyama24}, the stable category $\ul{\CM \!}^{\Z}_0 A$ contains a standard silting object
\[
V_A=\bigoplus_{i\in \mathbb{I}}\bigoplus_{s=1}^{-p_i +q_A}e_{\nu i}A(s)_{\geq 0}\: .
\]
The following theorem shows that the equivalence $F$ maps the standard silting object to the standard silting object up to direct summands.

\begin{proposition} \label{prop:standard silting}
Let $A$ and $V_A$ be as above. Let $F$ be the equivalence~(\ref{eq:covering}).
\begin{itemize}
\item[a)] The additive subcategories of $\ul{\CM \!}^{\Z}_0 A^{[m]}$ generated by $FV_A$ and $V_{A^{[m]}}$ coincide.
\item[b)] The endomophism algebras of the standard silting objects in $\ul{\CM \!}^{\Z}_0 A$ and $\ul{\CM \!}^{\Z}_0 A^{[m]}$ are Morita equivalent.
\end{itemize}
\end{proposition}

\begin{proof}
a) By definition, we have
\[
FV_A=\bigoplus_{i\in \mathbb{I}}\bigoplus_{s=1}^{-p_i +q_A}e_{\nu i, -s}A^{[m]}(\lceil \frac{s}{m}\rceil)_{\geq 0}\: .
\]
By Proposition~\ref{prop:covering of Gorenstein order}, we have
\[
V_{A^{[m]}}=\bigoplus_{i\in \mathbb{I}}\bigoplus_{j=0}^{m-1}\bigoplus_{s=1}^{-\lfloor \frac{p_i+j}{m}\rfloor +q_{A^{[m]}}}e_{\nu i, p_i+j}A^{[m]}(s)_{\geq 0}\: .
\]
Since both of them are silting objects in $\ul{\CM \!}^{\Z}_0 A^{[m]}$, they are maximal presilting objects. On the other hand, we have $mq_{A^{[m]}}\geq q_A$. This means that the object $FV_A$ is a direct summand of $V_{A^{[m]}}$. So the additive subcategories generated by them coincide.

b) This is a direct consequence of part~a).
\end{proof}

As we remarked in Section~\ref{ss:graded Morita equivalences}, any graded Morita equivalence is a composition of those induced by shifting a direct summand of the free graded module by degree $1$ or $-1$. The following proposition gives the compatibility of graded Morita equivalences and coverings.

\begin{proposition} \label{prop:compatibility}
Let $A$ be a $\Z$-graded algebra and $\{e_1, \ldots, e_n\}$ a complete set of primitive orthogonal idempotents of $A_0$. Let $P$ be the graded $A$-module $(A/e_{i_0}A)\oplus e_{i_0}A(1)$ and $B$ its graded endomorphism algebra. Then the $\Z$-graded algebra $B^{[m]}$ is isomorphic to the graded endomorphism algebra of the graded $A^{[m]}$-module $(A^{[m]}/e_{i_0, m-1}A^{[m]})\oplus e_{i_0, m-1}A^{[m]}(1)$.
\end{proposition}

\begin{proof}
The statement follows from direct calculation.
\end{proof}

\section{Standard tilting objects and their endomorphism algebras}

We prove the existence of a tilting object in the stable category $\ul{\CM \!}^{\Z} A$ for arbitrary $\N$-graded Gorenstein tiled orders $A$, extending Theorem~8.2 of \cite{IyamaKimuraUeyama24}. We also prove that for any finite poset, its incidence algebra can be realized as the endomorphism algebra of the standard tilting object in $\ul{\CM \!}^{\Z} A$ for an $\N$-graded Gorenstein tiled orders $A$ if and only if it is either empty or has the maximum.

\subsection{Standard tilting objects}

In this section, we prove that for any $\N$-graded Gorenstein tiled orders $A$, the object
\[
V=\begin{bmatrix}
R & \cdots & R
\end{bmatrix}\oplus \bigoplus_{i\in \mathbb{I}}\bigoplus_{j=1}^{-p_i}e_{\nu i}A(j)_{\geq 0}
\]
in $\ul{\CM \!}^{\Z} A$ is tilting. Moreover, its endomorphism algebra is isomorphic to the incidence algebra of a finite poset which is either empty or has the maximum.

\begin{theorem} \label{thm:tilting object}
Let $A$ be an $\N$-graded Gorenstein tiled order.
\begin{itemize}
\item[a)] The object $V$ in $\ul{\CM \!}^{\Z} A$ is tilting.
\item[b)] The endomorphism algebra of the standard tilting object $V$ in $\ul{\CM \!}^{\Z} A$ is isomorphic to $k\V_A\op$.
\item[c)] The poset $\V_A$ is either empty or a finite poset with minimum.
\end{itemize}
\end{theorem}

For example, let $A$ be as in Example~\ref{ex:poset}, then stable category $\ul{\CM \!}^\Z A$ is triangle equivalent to $\per k(A_{m-1}\times A_{n-1})$.
To prove Theorem~\ref{thm:tilting object}, we prepare the following lemmas.

\begin{lemma} \label{lem:tilting}
Let $A$ be an $\N$-graded Artin--Schelter Gorenstein algebra with Nakayama permutation $\nu$ and Gorenstein parameters $p_i\leq 1$, $i\in \mathbb{I}$. For any $i\in \mathbb{I}$, $j>0$, $1\leq j'\leq -p_i$, and $l>0$, we have
\[
\Hom_{\ul{\CM \!}^{\Z} A}(A(j)_{\geq 0}, \Si^{-l}e_{\nu i}A(j')_{\geq 0})=0 \: .
\]
\end{lemma}

\begin{proof}
Denote $M=\Omega^{l-1}(e_{\nu i}A(j')_{\geq 0})$ and $N=(A(j)/A(j)_{\geq 0})\ten_A \omega$. By the proof of part~(2) of Lemma~5.6 of \cite{IyamaKimuraUeyama24}, it suffices to show that we have $\Hom_k(Me_{\nu i'}, Ne_{\nu i'})_0=0$ for all $i'\in \mathbb{I}$.

We first consider the case that we have $p_{i'}\leq 0$. Since the graded vector space \linebreak $e_{\nu i}A(j')_{\geq 0}e_{\nu i'}$ is concentrated in non-negative degrees, so is the graded vector space $Me_{\nu i'}$. By part~(1) of Proposition~4.2 of \cite{IyamaKimuraUeyama24}, the graded vector space $Ne_{\nu i'}$ is concentrated in negative degrees. Therefore, we have $\Hom_k(Me_{\nu i'}, Ne_{\nu i'})_0=0$ in this case.

We now consider the case that we have $p_{i'}=1$. By part~(2) of Proposition~4.2 of \cite{IyamaKimuraUeyama24}, we have an isomorphism $e_{\nu i}A(j')e_{\nu i'} \simeq e_i A(j'-1+p_i)e_{i'}$ of graded vector spaces. This implies that the graded vector space $e_{\nu i}A(j')_{\geq 0}e_{\nu i'}$ is concentrated in positive degrees and hence so is the graded vector space $Me_{\nu i'}$.
By part~(1) of Proposition~4.2 of \cite{IyamaKimuraUeyama24}, the graded vector space $Ne_{\nu i'}$ is concentrated in non-positive degrees. Therefore, we also have $\Hom_k(Me_{\nu i'}, Ne_{\nu i'})_0=0$ in this case. This concludes the proof.
\end{proof}

\begin{lemma} \label{lem:basic}
Let $A$ be a basic $\N$-graded Gorenstein tiled order with Nakayama permutation $\nu$.
\begin{itemize}
\item[a)] At most one $b$-invariant of $A$ equals $0$.
\item[b)] For any $i\neq j$, not all of the integers $d(\nu i, i)$, $d(\nu i, j)$, $d(\nu j, i)$, $d(\nu j, j)$ are equal.
\end{itemize}
\end{lemma}

\begin{proof}
a) Assume that there are distinct $i$ and $j$ satisfying $d(\nu i, i)=d(\nu j, j)=0$, then we have $d(i, j)=d(\nu j, j)-d(\nu j, i)\leq d(\nu j, j)=0$. This implies that we have $d(i, j)=0$. Similarly, we have $d(j, i)=0$. This contradicts to that $A$ is basic.

b) If we have $d(\nu i, i)=d(\nu i, j)=d(\nu j, i)=d(\nu j, j)$, then it follows that we have $d(i, j)=d(\nu j, j)-d(\nu j, i)=0$ and $d(j, i)=d(\nu i, i)-d(\nu i, j)=0$. This contradicts to that $A$ is basic.
\end{proof}

\begin{lemma} \label{lem:maximum and minimum}
Let $A$ be a basic $\N$-graded Gorenstein tiled order. If $A$ has a zero $b$-invariant, then the poset $\V_A$ is either empty, or has the maximum and the minimum.
\end{lemma}

\begin{proof}
Denote the exponent matrix associated with the tiled order $A$ by ${\rm v}(A)=[d(i, j)]$. Assume that the poset $\V_A$ is nonempty and the $b$-invariant $d(\nu i, i)$ equals $0$, we prove that it has the maximum and the minimum. Since all entries in the $i$-th column of ${\rm v}(A)$ equal $0$, we have
\[
d(i, j)=d(\nu j, j)-d(\nu j, i)=d(\nu j, j)\geq d(k, j)
\]
for all $j$, $k\in \mathbb{I}$. This implies that the element ${\rm v}(e_i A(1)_{\geq 0})$ is the maximum of $\V_A$. Since all entries in the $\nu i$-th row of ${\rm v}(A)$ equal $0$, we have
\[
d(\nu^2 i, \nu i)\geq d(j, \nu i)=d(j, \nu^{-1}j)-d(\nu i, \nu^{-1}j)=d(j, \nu^{-1}j)
\]
for all $j\in \mathbb{I}$. This means that $d(\nu^2 i, \nu i)$ is the greatest $b$-invariant of $A$. If it is less than or equal to $1$, then so are the others. This contradicts to that $\V_A$ is nonempty. So we have $d(\nu^2 i, \nu i)\geq 2$. By Lemma~\ref{lem:basic}, the $b$-invariant $d(\nu j, j)$ is positive for all $j\neq i$. Then we have
\[
d(\nu^2 i, j)=d(\nu^2 i,\nu i)-d(j, \nu i)=d(\nu^2 i,\nu i)-d(j, \nu^{-1}j)<d(\nu^2 i, \nu i)
\]
for all $j\neq \nu i$. We deduce that we have ${\rm v}(e_{\nu^2 i}A(d(\nu^2 i, \nu i)-1)_{\geq 0})=e_{\nu i}$. On the other hand, since we have $d(j, \nu i)=d(j, \nu^{-1}j)\geq d(j, k)$ for all $j$, $k\in \mathbb{I}$, any nonzero element ${\rm v}(e_j A(l)_{\geq 0})\in \V_A$ is greater than or equal to $e_{\nu i}$. This means that ${\rm v}(e_{\nu^2 i}A(d(\nu^2 i, \nu i)-1)_{\geq 0})$ is the minimum of $\V_A$.
\end{proof}

\begin{proof}[Proof of Theorem \ref{thm:tilting object}]
After possibly replacing $A$ with a graded Morita equivalent \linebreak $\N$-graded Gorenstein tiled order, we may and will assume that it is basic.

a) If all $b$-invariants of $A$ are positive, then the statement follows from part~(1) of Theorem~8.2 of \cite{IyamaKimuraUeyama24}. We now assume that $A$ has a zero $b$-invariant. So we have
\[
V=\bigoplus_{i\in \mathbb{I}}\bigoplus_{j=1}^{-p_i}e_{\nu i}A(j)_{\geq 0}\: .
\]
Since $A$ is an $\N$-graded Gorenstein tiled order, its Gorenstein parameters are less than or equal to $1$. Then the statement follows from part~(1) of Theorem~5.1 of \cite{IyamaKimuraUeyama24} and Lemma~\ref{lem:tilting}.

b) This follows from the proof of part~(3) of Theorem~8.2 of \cite{IyamaKimuraUeyama24}.

c) If all $b$-invariants of $A$ are positive, then $\V_A$ has the minimum $0$. If $A$ has a zero $b$-invariant, by Lemma~\ref{lem:maximum and minimum}, the poset $\V_A$ is either empty, or has the minimum.
\end{proof}

\subsection{Realization of posets} \label{ss:realization of posets}

In this section, we prove that for a finite poset with maximum, its incidence algebra can be realized as the endomorphism algebra of the standard tilting object $V_A$ for an $\N$-graded Gorenstein tiled order $A$.

\begin{theorem} \label{thm:realization of posets}
Let $P$ be a finite poset with maximum. Then there exists a basic $\N$-graded Gorenstein tiled order $A$ with all $b$-invariants $1$ or $2$ such that the endomorphism algebra of the standard tilting object $V$ in $\ul{\CM \!}^{\Z} A$ is isomorphic to $kP$.
\end{theorem}

\begin{corollary} \label{cor:realization of quivers}
Let $Q$ be a finite quiver. Then the following are equivalent.
\begin{itemize}
\item[i)] Either $Q$ vanishes, or the underlying graph of $Q$ is a tree and $Q$ has a unique sink.
\item[ii)] There exists an $\N$-graded Gorenstein tiled order $A$ such that the endomorphism algebra of the standard tilting object $V$ in $\ul{\CM \!}^{\Z} A$ is isomorphic to $kQ$.
\end{itemize}
\end{corollary}

\begin{proof}
Let us show that i) implies ii): If $Q$ vanishes, then we let $A$ be $k[x]$. If the underlying graph of $Q$ is a tree and $Q$ has a unique sink, then the statement follows from Theorem~\ref{thm:realization of posets}.

We now show that ii) implies i): If $Q$ does not vanish, by parts~b) and c) of Theorem~\ref{thm:tilting object}, there is a finite poset $P$ with maximum such that $kP$ is isomorphic to $kQ$. It follows that $kP$ is hereditary and hence the underlying graph of the Hasse quiver of $P$ is a tree.
\end{proof}

\begin{proof}[Proof of Theorem~\ref{thm:realization of posets}]
Denote the maximum of $P$ by $0$. We extend the finite poset $P\setminus \{ 0\}$ to a totally ordered set $T$. We define the matrix $d=[d(i, j)]_{i, j\in T}$ as $d(i, j)=0$ if $i \leq j$ in $P\setminus \{ 0\}$, otherwise $d(i, j)=1$. We claim that the matrix
\[
{\rm v}(A)=
\begin{bmatrix}
d & \mathbf{1}-d^t \\
\mathbf{2}-d^t & d
\end{bmatrix}
\]
gives rise to a tiled order $A$ which satisfies the required conditions, where $\mathbf{1}$ denotes the matrix with all entries $1$.

We first check that $A$ is an $R$-algebra. For any $i$, $j$, $k\in T$, if $d(i, k)$ or $d(k, j)$ equals $1$, then we have $d(i, k)+d(k, j)\geq 1 \geq d(i, j)$. If $d(i, k)$ and $d(k, j)$ equal $0$, then we have $i\leq k$ and $k\leq j$ in $P\setminus \{ 0\}$. So we have $i\leq j$ and hence $d(i, j)$ equals $0$. Therefore, we have $d(i, k)+d(k, j)\geq d(i, j)$ for all $i$, $j$, $k\in T$. Using this and that all entries of $\mathbf{2}-d^t$ is greater than or equal to $1$ we deduce that the triangle inequality holds for the matrix ${\rm v}(A)$.

By definition, we have $d(i, i)=0$ and $d(i, j)=1$ if $i>j$ in $T$. It follows that the difference of each two rows of $d$ is not a constant vector and hence the same holds for the matrix ${\rm v}(A)$. This means that $A$ is basic. By construction, the entries in the diagonal of the blocks $\mathbf{1}-d^t$ and $\mathbf{2}-d^t$ are the $b$-invariants of $A$. So $A$ is a Gorenstein order with all $b$-invariants $1$ or $2$. For any $i$, $j\in T$, we have $e_i(\mathbf{2}-d^t)\geq e_j(\mathbf{2}-d^t)$ if and only if $d(k, i)\leq d(k, j)$ for all $k\in T$ if and only if $i\geq j$ in $P\setminus \{ 0\}$. By part~(3) of Theorem~8.2 of \cite{IyamaKimuraUeyama24}, the endomorphism algebra of $V$ in $\ul{\CM \!}^{\Z} A$ is isomorphic to $kP$. This concludes the proof.
\end{proof}

The following example serves to illustrate the construction of tiled orders $A$.

\begin{example}
Let $Q$ be a quiver of the following shape.
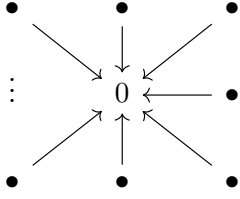
\[
\begin{tikzcd}
\bullet \arrow{dr} & \bullet \arrow{d} & \bullet \arrow{dl} \\[-10pt]
\vdots & 0 & \bullet \arrow{l} \\
\bullet \arrow{ur} & \bullet \arrow{u} & \bullet \arrow{ul}
\end{tikzcd}
\]
Then $Q\setminus \{ 0\}$ has no arrows. By definition, we have $d=\mathbf{1}-I$. Therefore, the matrix
\[
{\rm v}(A)=
\begin{bsmallmatrix}
0 & 1 & \cdots & 1 & 1 & 1 & 0 & \cdots & 0 & 0 \\
1 & 0 & \cdots & 1 & 1 & 0 & 1 & \cdots & 0 & 0 \\[-5pt]
\vdots & \vdots & \ddots & \vdots & \vdots & \vdots & \vdots & \ddots & \vdots & \vdots \\
1 & 1 & \cdots & 0 & 1 & 0 & 0 & \cdots & 1 & 0 \\
1 & 1 & \cdots & 1 & 0 & 0 & 0 & \cdots & 0 & 1 \\
2 & 1 & \cdots & 1 & 1 & 0 & 1 & \cdots & 1 & 1 \\
1 & 2 & \cdots & 1 & 1 & 1 & 0 & \cdots & 1 & 1 \\[-5pt]
\vdots & \vdots & \ddots & \vdots & \vdots & \vdots & \vdots & \ddots & \vdots & \vdots \\
1 & 1 & \cdots & 2 & 1 & 1 & 1 & \cdots & 0 & 1 \\
1 & 1 & \cdots & 1 & 2 & 1 & 1 & \cdots & 1 & 0
\end{bsmallmatrix}
\]
gives rise to a tiled order $A$ such that the endomorphism algebra of the standard tilting object $V$ in $\ul{\CM \!}^{\Z} A$ is isomorphic to $kQ$.
\end{example}

\section{Classification of Gorenstein tiled orders}

In Theorem~\ref{thm:realization of posets}, we have realized all possible incidence algebras of finite posets as the endomorphism algebra $\Ga$ of the standard tilting object in $\ul{\CM \!}^{\Z} A$ for a basic $\N$-graded Gorenstein tiled order $A$. A natural question is to classify the basic $\N$-graded Gorenstein tiled orders such that the endomorphism algebra $\Ga$ is isomorphic to $kP$ for a given finite poset $P$. We give the answer for suitable finite posets $P$.

\subsection{Type $A_n$ with $n\leq 2$}

\begin{theorem} \label{thm:classification 1}
Let $A$ be a basic $\N$-graded Gorenstein tiled order.
\begin{itemize}
\item[a)] The algebra $\Ga$ vanishes if and only if $A$ is hereditary if and only if $A$ is conjugate to the tiled order determined by the first exponent matrix below of arbitrary size.
\item[b)] The algebra $\Ga$ is isomorphic to $kA_1$ if and only if $A$ is conjugate to the tiled order determined by the second exponent matrix below of even size.
\item[c)] The algebra $\Ga$ is isomorphic to $kA_2$ if and only if $A$ is conjugate to the tiled order determined by the third exponent matrix below of size greater than or equal to $2$.
\end{itemize}
\[
\begin{bsmallmatrix}
0 & 0 & \cdots & 0 & 0 \\
1 & 0 & \cdots & 0 & 0 \\[-5pt]
\vdots & \vdots & \ddots & \vdots & \vdots \\
1 & 1 & \cdots & 0 & 0 \\
1 & 1 & \cdots & 1 & 0
\end{bsmallmatrix}
\quad
\begin{bsmallmatrix}
0 & 1 & 0 & 0 & \cdots & 0 & 0 \\
1 & 0 & 0 & 0 & \cdots & 0 & 0 \\
1 & 1 & 0 & 1 & \cdots & 0 & 0 \\
1 & 1 & 1 & 0 & \cdots & 0 & 0 \\[-5pt]
\vdots & \vdots & \vdots & \vdots & \ddots & \vdots & \vdots \\
1 & 1 & 1 & 1 & \cdots & 0 & 1 \\
1 & 1 & 1 & 1 & \cdots & 1 & 0
\end{bsmallmatrix}
\quad
\begin{bsmallmatrix}
0 & 1 & 1 & \cdots & 1 & 2 \\
1 & 0 & 1 & \cdots & 1 & 1 \\
0 & 1 & 0 & \cdots & 1 & 1 \\[-5pt]
\vdots & \vdots & \vdots & \ddots & \vdots & \vdots \\
0 & 0 & 0 & \cdots & 0 & 1 \\
0 & 0 & 0 & \cdots & 1 & 0
\end{bsmallmatrix}
\]
\end{theorem}

To prove Theorem~\ref{thm:classification 1}, we need the following lemmas.

\begin{lemma} \label{lem:cyclic Gorenstein tiled order}
Let $A$ be a basic Gorenstein tiled order. Then $A$ is isomorphic to a tiled order whose associated exponent matrix can be divided into blocks such that the diagonal blocks are the associated exponent matrices of cyclic Gorenstein tiled orders.
\end{lemma}

\begin{proof}
Let the cardinalities of the orbits of $\mathbb{I}$ under the Nakayama permutation be \linebreak $n_1$, \ldots, $n_j$. After possibly conjugating $A$ by a permutation matrix, we may and will assume that the orbits of $\mathbb{I}$ under the Nakayama permutation are
\[
\{1, \ldots, n_1\} \ko \ldots \ko \{\sum_{i=1}^{j-1}n_i+1, \ldots, \sum_{i=1}^j n_i\} \: .
\]
Then the exponent matrix associated with $A$ satisfies the required condition.
\end{proof}

\begin{lemma} \label{lem:gluing}
Let $A$ be a basic Gorenstein tiled order. Suppose that the matrix ${\rm v}(A)$ is a block matrix of blockwise size $2$ such that the two diagonal blocks are the exponent matrices associated with Gorenstein tiled orders $A_1$ and $A_2$ with $b$-invariants $b_i$ respectively $b'_j$. Then we have $\frac{1}{|\mathbb{I}_{A_1}|}\sum_{i\in \mathbb{I}_{A_1}} b_i=\frac{1}{|\mathbb{I}_{A_2}|}\sum_{j\in \mathbb{I}_{A_2}}b'_j$.
\end{lemma}

\begin{proof}
The statement follows from Proposition~4.4 of \cite{IyamaKimuraUeyama24}.
\end{proof}

For any basic $\N$-graded Gorenstein tiled order $A$ which has a zero $b$-invariant $d(\nu i, i)$ and the poset $\V_A$ is nonempty, by the proof of Lemma~\ref{lem:maximum and minimum}, the element ${\rm v}(e_i A(1)_{\geq 0})$ is the maximum in $\V_A$ and we have $d(i, \nu^{-1}i)\geq 2$. So if we replace $A$ by $DAD^{-1}$, where $D$ is the diagonal matrix with the $i$-th diagonal entry $x^{-1}$ and the others $1$, then all $b$-invariants become positive and $\V_{DAD^{-1}}$ is obtained from $\V_A$ by removing the maximum and adding the new minimum $0$, \cf~Proposition~\ref{prop:conjugacy}. In particular, the poset does not change if it is totally ordered. Therefore, to classify basic $\N$-graded Gorenstein tiled orders up to conjugacies such that the endomorphism algebra $\Ga$ of the standard tilting object is isomorphic to the incidence algebra of a given nonempty totally ordered set, without loss of generality, we may assume that all $b$-invariants of $A$ are positive.

\begin{lemma} \label{lem:cyclic Gorenstein tiled orders}
Let $A$ be a basic $\N$-graded cyclic Gorenstein tiled order such that the stable category $\ul{\CM \!}^{\Z} A$ is triangle equivalent to $\per kA_1$. Then $A$ is conjugate to a tiled order determined by the exponent matrix
$\begin{bsmallmatrix}
0 & 1 \\
1 & 0
\end{bsmallmatrix}$.
\end{lemma}

\begin{proof}
Without loss of generality, we may and will assume that all $b$-invariants of $A$ are positive. Since the poset $\V_A$ has the single element $0$, all $b$-invariants of $A$ must be $1$. This implies that all entries of ${\rm v}(A)$ are $0$ or $1$. Assume that we have $\mathbb{I}=\{1, \ldots, n\}$. We extend the definition of $\mathbb{I}$ to $\Z$ by identifying $i+tn$ with $i$ for all $i\in \mathbb{I}$ and integers $t$. It suffices to show that $n$ equals $2$.

Assume that we have $n\geq 3$. Since $A$ is a cyclic Gorenstein tiled order, after possibly conjugating $A$ by a permutation matrix, we may and will assume that we have $\nu i=i+1$ for all $i\in \mathbb{I}$. Since we have $\sum_{i=1}^{n-1}d(i, i+1)\geq d(1, n)=1$, one of the summands equals $1$. If we use the equalities $d(i+1, j)+d(j, i)=d(i+1, i)$ iteratively, we deduce that we have $d(i, i+1)=1$ and $d(i+2, i)=0$ for all $i\in \mathbb{I}$. If $n$ equals $3$, then this is a contradiction. If $n$ equals $4$, then we have $d(1, 3)=d(3, 1)=0$. This contradicts to the basicness of $A$. Assume that we have $n\geq 5$. Since we have $d(i+2, i)=0$ for all $i\in \mathbb{I}$, by the basicness of $A$, we have $d(i, i+2)=1$ for all $i\in \mathbb{I}$. If $n$ equals $5$, then we have $d(1, 5)=d(1, 3)+d(3, 5)=2$. This contradicts to that it equals $1$. Assume that we have $n\geq 6$. We consider the equalities $d(i+3, i)=d(i+3, i+2)-d(i, i+2)=0$ for all $i\in \mathbb{I}$. If $n$ is odd, then we have $\sum_{i=1}^{\frac{n-1}{2}} d(2i+1, 2i-1)\geq d(n, 1)$. If $n$ is even, then we have $\sum_{i=1}^{\frac{n}{2}-2} d(2i+1, 2i-1)+d(n, n-3)\geq d(n, 1)$. So we deduce that $d(n, 1)$ must be $0$. This contradicts to that it equals $1$. This concludes the proof.
\end{proof}

\begin{proof}[Proof of Theorem \ref{thm:classification 1}]
The sufficiency in each part follows from part~b) of Theorem~\ref{thm:tilting object}. Let us prove the necessity. Assume that we have $\mathbb{I}=\{1, \ldots, n\}$. We extend the definition of $\mathbb{I}$ to $\Z$ by identifying $i+tn$ with $i$ for all $i\in \mathbb{I}$ and integers $t$.

a) Since $A$ is a Gorenstein tiled $k[x]$-order and the stable category $\ul{\CM \!}^{\Z} A$ vanishes, the algebra $A$ is hereditary. If all $b$-invariants of $A$ are positive or $A$ has a $b$-invariant greater than $1$, then by part~(3) of Theorem~8.2 of \cite{IyamaKimuraUeyama24}, the stable category $\ul{\CM \!}^{\Z} A$ does not vanish. This contradicts to that the global dimension of $A$ is finite. Then by Lemma~\ref{lem:basic}, exactly one of the $b$-invariants of $A$ equals $0$ and the others equal $1$. By Lemmas~\ref{lem:cyclic Gorenstein tiled order} and \ref{lem:gluing}, the Gorenstein order $A$ must be cyclic. After possibly conjugating $A$ by a permutation matrix, we may and will assume that we have $\nu i=i+1$ for all $i\in \mathbb{I}$ and $d(1, n)=0$. This implies that we have $d(1, i)=d(i, n)=0$ for all $i\in \mathbb{I}$. If we use the equalities $d(i+1, j)+d(j, i)=d(i+1, i)$ iteratively, all entries of the matrix ${\rm v}(A)$ are determined to be the desired one.

b) Without loss of generality, we may and will assume that all $b$-invariants of $A$ are positive. Since the poset $\V_A$ has the single element $0$, all $b$-invariants of $A$ must be $1$. This implies that all entries of ${\rm v}(A)$ are $0$ or $1$. By Lemma~\ref{lem:cyclic Gorenstein tiled order}, we may and will assume that the exponent matrix ${\rm v}(A)$ is a block matrix with all diagonal blocks of the form in Lemma~\ref{lem:cyclic Gorenstein tiled orders}. For distinct $i$ and $j$, by the Gorensteinness and the basicness of $A$, one of the $(i, j)$-block and the $(j, i)$-block of ${\rm v}(A)$ is the constant matrix $\mathbf{0}$ and the other is the constant matrix $\mathbf{1}$. We prove that $A$ is conjugate to the desired form by induction on the blockwise size of ${\rm v}(A)$. If the blockwise size of ${\rm v}(A)$ equals $1$, then it is clear. Assume that the claim holds if the blockwise size of ${\rm v}(A)$ equals $m\geq 1$. If the blockwise size of ${\rm v}(A)$ is $m+1$, then by induction, we may and will assume that its upper-left block of blockwise size $m$ is the desired form. If all blocks in the $(m+1)$-th blockwise column are the constant matrix $\mathbf{0}$, then we are done. Otherwise assume that the first block in the $(m+1)$-th blockwise column which is the constant matrix $\mathbf{1}$ lies in the $i$-th blockwise row. By the triangle inequalities, the $(j, m+1)$-block is also the constant matrix $\mathbf{1}$ for all $i<j\leq m$. If we conjugate $A$ by a permutation matrix, we can switch the $(i, m+1)$-block and the $(m+1, i)$-block of ${\rm v}(A)$ and keep the others. Then we do this for $i+1$, \ldots, $m+1$ to end the induction. This concludes the proof.

c) Without loss of generality, we may and will assume that all $b$-invariants of $A$ are positive. Since the poset $\V_A$ has only two elements, all $b$-invariants of $A$ must be less than or equal to $2$. If two $b$-invariants $d(\nu i, i)$ and $d(\nu j, j)$ of $A$ equal $2$, then both elements ${\rm v}(e_{\nu i}A(1)_{\geq 0})$ and ${\rm v}(e_{\nu j}A(1)_{\geq 0})$ of $\V_A$ are nonzero. So they must coincide. It follows that the entries $d(\nu i, j)$ and $d(\nu j, i)$ also equal $2$. By Lemma~\ref{lem:basic}, this is a contradiction. Therefore, exactly one of the $b$-invariants of $A$ equals $2$ and the others equal $1$. By Lemmas~\ref{lem:cyclic Gorenstein tiled order} and \ref{lem:gluing}, the Gorenstein order $A$ must be cyclic. After possibly conjugating $A$ by a permutation matrix, we may and will assume that we have $\nu i=i+1$ for all $i\in \mathbb{I}$ and $d(1, n)=2$. Since all the other $b$-invariants equal $1$, we must have $d(1, i)=d(i, n)=1$ for all $i\neq 1$, $n$. If we use the equalities $d(i+1, j)+d(j, i)=d(i+1, i)$ iteratively, all entries of the matrix ${\rm v}(A)$ are determined to be the desired one.
\end{proof}

\subsection{Type $A_3$}

\begin{theorem} \label{thm:classification 2}
Let $A$ be a basic $\N$-graded Gorenstein tiled order.
\begin{itemize}
\item[a)] The algebra $\Ga$ is isomorphic to the path algebra of the non-linear $A_3$ quiver with unique sink if and only if $A$ is isomorphic to the tiled order determined by the exponent matrix
$\begin{bsmallmatrix}
d_1 & d_2 \\
d_2 & d_1
\end{bsmallmatrix}$ of the following types.
\[
\begin{array}{|c||c|c|c|c|c|}\hline
\mbox{type} & d_1 & d_2 & \mbox{size of $d_i$}\\ \hline
\mathrm{I} &
\begin{bsmallmatrix}
0 & 1 & 1 & \cdots & 1 & 2 \\
1 & 0 & 1 & \cdots & 1 & 1 \\
0 & 1 & 0 & \cdots & 1 & 1 \\[-5pt]
\vdots & \vdots & \vdots & \ddots & \vdots & \vdots \\
0 & 0 & 0 & \cdots & 0 & 1 \\
0 & 0 & 0 & \cdots & 1 & 0
\end{bsmallmatrix} &
\begin{bsmallmatrix}
1 & 1 & 1 & \cdots & 1 & 1 \\
0 & 1 & 1 & \cdots & 1 & 1 \\
0 & 0 & 1 & \cdots & 1 & 1 \\[-5pt]
\vdots & \vdots & \vdots & \ddots & \vdots & \vdots \\
0 & 0 & 0 & \cdots & 1 & 1 \\
0 & 0 & 0 & \cdots & 0 & 1
\end{bsmallmatrix} & \geq 2 \\ \hline
\multirow{2}{*}{\raisebox{-3\totalheight}{$\mathrm{II}$}}
 &
\begin{bsmallmatrix}
0
\end{bsmallmatrix} &
\begin{bsmallmatrix}
2
\end{bsmallmatrix} & 1 \\ \cline{2-4}
 &
\begin{bsmallmatrix}
0 & 1 & 1 & \cdots & 1 & 1 \\
1 & 0 & 1 & \cdots & 1 & 1 \\
0 & 1 & 0 & \cdots & 1 & 1 \\[-5pt]
\vdots & \vdots & \vdots & \ddots & \vdots & \vdots \\
0 & 0 & 0 & \cdots & 0 & 1 \\
0 & 0 & 0 & \cdots & 1 & 0
\end{bsmallmatrix} &
\begin{bsmallmatrix}
1 & 1 & 1 & \cdots & 1 & 2 \\
0 & 1 & 1 & \cdots & 1 & 1 \\
0 & 0 & 1 & \cdots & 1 & 1 \\[-5pt]
\vdots & \vdots & \vdots & \ddots & \vdots & \vdots \\
0 & 0 & 0 & \cdots & 1 & 1 \\
0 & 0 & 0 & \cdots & 0 & 1
\end{bsmallmatrix} & \geq 2 \\ \hline
\end{array}
\]
\item[b)] The algebra $\Ga$ is isomorphic to the path algebra of the linear $A_3$ quiver if and only if $A$ is isomorphic to the tiled order determined by the exponent matrix
$\begin{bsmallmatrix}
d_1 & d_2 \\
d_3 & d_4
\end{bsmallmatrix}$ of the following types.
\[
\adjustbox{max width=\textwidth}{
$\begin{array}{|c||c|c|c|c|c|}\hline
\mbox{type} & d_1 & d_2 & d_3 & d_4 & \mbox{size of $d_i$}\\ \hline
\mathrm{III} &
\begin{bsmallmatrix}
0 & 1 & 1 & \cdots & 1 & 1 & 2 \\
1 & 0 & 1 & \cdots & 1 & 1 & 1 \\
0 & 1 & 0 & \cdots & 1 & 1 & 1 \\[-5pt]
\vdots & \vdots & \vdots & \ddots & \vdots & \vdots & \vdots \\
0 & 0 & 0 & \cdots & 0 & 1 & 1 \\
0 & 0 & 0 & \cdots & 1 & 0 & 1 \\
0 & 0 & 0 & \cdots & 0 & 1 & 0
\end{bsmallmatrix} &
\begin{bsmallmatrix}
0 & 1 & 1 & \cdots & 1 & 1 & 1 \\
0 & 0 & 1 & \cdots & 1 & 1 & 1 \\
0 & 0 & 0 & \cdots & 1 & 1 & 1 \\[-5pt]
\vdots & \vdots & \vdots & \ddots & \vdots & \vdots & \vdots \\
0 & 0 & 0 & \cdots & 0 & 1 & 1 \\
0 & 0 & 0 & \cdots & 0 & 0 & 1 \\
0 & 0 & 0 & \cdots & 0 & 0 & 0
\end{bsmallmatrix} &
\begin{bsmallmatrix}
1 & 1 & 1 & \cdots & 1 & 1 & 2 \\
1 & 1 & 1 & \cdots & 1 & 1 & 1 \\
0 & 1 & 1 & \cdots & 1 & 1 & 1 \\[-5pt]
\vdots & \vdots & \vdots & \ddots & \vdots & \vdots & \vdots \\
0 & 0 & 0 & \cdots & 1 & 1 & 1 \\
0 & 0 & 0 & \cdots & 1 & 1 & 1 \\
0 & 0 & 0 & \cdots & 0 & 1 & 1
\end{bsmallmatrix} &
\begin{bsmallmatrix}
0 & 1 & 1 & \cdots & 1 & 1 & 2 \\
1 & 0 & 1 & \cdots & 1 & 1 & 1 \\
0 & 1 & 0 & \cdots & 1 & 1 & 1 \\[-5pt]
\vdots & \vdots & \vdots & \ddots & \vdots & \vdots & \vdots \\
0 & 0 & 0 & \cdots & 0 & 1 & 1 \\
0 & 0 & 0 & \cdots & 1 & 0 & 1 \\
0 & 0 & 0 & \cdots & 0 & 1 & 0
\end{bsmallmatrix} & \geq 2 \\ \hline
\multirow{2}{*}{\raisebox{-3.8\totalheight}{$\mathrm{IV}$}}
 &
\begin{bsmallmatrix}
0 & 2 \\
1 & 0
\end{bsmallmatrix} &
\begin{bsmallmatrix}
2 & 0 \\
1 & 0
\end{bsmallmatrix} &
\begin{bsmallmatrix}
0 & 0 \\
1 & 2
\end{bsmallmatrix} &
\begin{bsmallmatrix}
0 & 0 \\
3 & 0
\end{bsmallmatrix} & 2 \\ \cline{2-6}
 &
\begin{bsmallmatrix}
0 & 1 & 1 & 1 & \cdots & 1 & 1 & 2 \\
1 & 0 & 1 & 1 & \cdots & 1 & 1 & 1 \\
0 & 1 & 0 & 1 & \cdots & 1 & 1 & 1 \\
0 & 0 & 1 & 0 & \cdots & 1 & 1 & 1 \\[-5pt]
\vdots & \vdots & \vdots & \vdots & \ddots & \vdots & \vdots & \vdots \\
0 & 0 & 0 & 0 & \cdots & 0 & 1 & 1 \\
0 & 0 & 0 & 0 & \cdots & 1 & 0 & 1 \\
0 & 0 & 0 & 0 & \cdots & 0 & 1 & 0
\end{bsmallmatrix} &
\begin{bsmallmatrix}
2 & 1 & 1 & 1 & \cdots & 1 & 1 & 0 \\
1 & 1 & 1 & 1 & \cdots & 1 & 1 & 0 \\
1 & 0 & 1 & 1 & \cdots & 1 & 1 & 0 \\
1 & 0 & 0 & 1 & \cdots & 1 & 1 & 0 \\[-5pt]
\vdots & \vdots & \vdots & \vdots & \ddots & \vdots & \vdots & \vdots \\
1 & 0 & 0 & 0 & \cdots & 1 & 1 & 0 \\
1 & 0 & 0 & 0 & \cdots & 0 & 1 & 0 \\
1 & 0 & 0 & 0 & \cdots & 0 & 0 & 0
\end{bsmallmatrix} &
\begin{bsmallmatrix}
0 & 0 & 0 & 0 & \cdots & 0 & 0 & 0 \\
0 & 1 & 1 & 1 & \cdots & 1 & 1 & 1 \\
0 & 0 & 1 & 1 & \cdots & 1 & 1 & 1 \\
0 & 0 & 0 & 1 & \cdots & 1 & 1 & 1 \\[-5pt]
\vdots & \vdots & \vdots & \vdots & \ddots & \vdots & \vdots & \vdots \\
0 & 0 & 0 & 0 & \cdots & 1 & 1 & 1 \\
0 & 0 & 0 & 0 & \cdots & 0 & 1 & 1 \\
1 & 1 & 1 & 1 & \cdots & 1 & 1 & 2
\end{bsmallmatrix} &
\begin{bsmallmatrix}
0 & 0 & 0 & 0 & \cdots & 0 & 0 & 0 \\
2 & 0 & 1 & 1 & \cdots & 1 & 1 & 0 \\
1 & 1 & 0 & 1 & \cdots & 1 & 1 & 0 \\
1 & 0 & 1 & 0 & \cdots & 1 & 1 & 0 \\[-5pt]
\vdots & \vdots & \vdots & \vdots & \ddots & \vdots & \vdots & \vdots \\
1 & 0 & 0 & 0 & \cdots & 0 & 1 & 0 \\
1 & 0 & 0 & 0 & \cdots & 1 & 0 & 0 \\
2 & 1 & 1 & 1 & \cdots & 1 & 2 & 0
\end{bsmallmatrix} & \geq 3 \\ \hline
\multirow{2}{*}{\raisebox{-3.4\totalheight}{$\mathrm{V}$}}
 &
\begin{bsmallmatrix}
0
\end{bsmallmatrix} &
\begin{bsmallmatrix}
1
\end{bsmallmatrix} &
\begin{bsmallmatrix}
3
\end{bsmallmatrix} &
\begin{bsmallmatrix}
0
\end{bsmallmatrix} & 1 \\ \cline{2-6}
 &
\begin{bsmallmatrix}
0 & 1 & 1 & \cdots & 1 & 1 & 1 \\
1 & 0 & 1 & \cdots & 1 & 1 & 1 \\
0 & 1 & 0 & \cdots & 1 & 1 & 1 \\[-5pt]
\vdots & \vdots & \vdots & \ddots & \vdots & \vdots & \vdots \\
0 & 0 & 0 & \cdots & 0 & 1 & 1 \\
0 & 0 & 0 & \cdots & 1 & 0 & 1 \\
0 & 0 & 0 & \cdots & 0 & 1 & 0
\end{bsmallmatrix} &
\begin{bsmallmatrix}
1 & 1 & 1 & \cdots & 1 & 1 & 1 \\
0 & 1 & 1 & \cdots & 1 & 1 & 0 \\
0 & 0 & 1 & \cdots & 1 & 1 & 0 \\[-5pt]
\vdots & \vdots & \vdots & \ddots & \vdots & \vdots & \vdots \\
0 & 0 & 0 & \cdots & 1 & 1 & 0 \\
0 & 0 & 0 & \cdots & 0 & 1 & 0 \\
0 & 0 & 0 & \cdots & 0 & 0 & 0
\end{bsmallmatrix} &
\begin{bsmallmatrix}
1 & 1 & 1 & \cdots & 1 & 1 & 2 \\
0 & 1 & 1 & \cdots & 1 & 1 & 1 \\
0 & 0 & 1 & \cdots & 1 & 1 & 1 \\[-5pt]
\vdots & \vdots & \vdots & \ddots & \vdots & \vdots & \vdots \\
0 & 0 & 0 & \cdots & 1 & 1 & 1 \\
0 & 0 & 0 & \cdots & 0 & 1 & 1 \\
1 & 1 & 1 & \cdots & 1 & 1 & 2
\end{bsmallmatrix} &
\begin{bsmallmatrix}
0 & 1 & 1 & \cdots & 1 & 1 & 0 \\
1 & 0 & 1 & \cdots & 1 & 1 & 0 \\
0 & 1 & 0 & \cdots & 1 & 1 & 0 \\[-5pt]
\vdots & \vdots & \vdots & \ddots & \vdots & \vdots & \vdots \\
0 & 0 & 0 & \cdots & 0 & 1 & 0 \\
0 & 0 & 0 & \cdots & 1 & 0 & 0 \\
1 & 1 & 1 & \cdots & 1 & 2 & 0
\end{bsmallmatrix} & \geq 2 \\ \hline
\multirow{2}{*}{\raisebox{-3.8\totalheight}{$\mathrm{VI}$}}
 &
\begin{bsmallmatrix}
0
\end{bsmallmatrix} &
\begin{bsmallmatrix}
0
\end{bsmallmatrix} &
\begin{bsmallmatrix}
4
\end{bsmallmatrix} &
\begin{bsmallmatrix}
0
\end{bsmallmatrix} & 1 \\ \cline{2-6}
 &
\begin{bsmallmatrix}
0 & 0 & 0 & 0 & \cdots & 0 & 0 & 0 \\
2 & 0 & 1 & 1 & \cdots & 1 & 1 & 1 \\
1 & 1 & 0 & 1 & \cdots & 1 & 1 & 1 \\
1 & 0 & 1 & 0 & \cdots & 1 & 1 & 1 \\[-5pt]
\vdots & \vdots & \vdots & \vdots & \ddots & \vdots & \vdots & \vdots \\
1 & 0 & 0 & 0 & \cdots & 0 & 1 & 1 \\
1 & 0 & 0 & 0 & \cdots & 1 & 0 & 1 \\
1 & 0 & 0 & 0 & \cdots & 0 & 1 & 0
\end{bsmallmatrix} &
\begin{bsmallmatrix}
0 & 0 & 0 & 0 & \cdots & 0 & 0 & 0 \\
0 & 1 & 1 & 1 & \cdots & 1 & 1 & 0 \\
0 & 0 & 1 & 1 & \cdots & 1 & 1 & 0 \\
0 & 0 & 0 & 1 & \cdots & 1 & 1 & 0 \\[-5pt]
\vdots & \vdots & \vdots & \vdots & \ddots & \vdots & \vdots & \vdots \\
0 & 0 & 0 & 0 & \cdots & 1 & 1 & 0 \\
0 & 0 & 0 & 0 & \cdots & 0 & 1 & 0 \\
0 & 0 & 0 & 0 & \cdots & 0 & 0 & 0
\end{bsmallmatrix} &
\begin{bsmallmatrix}
2 & 1 & 1 & 1 & \cdots & 1 & 1 & 2 \\
1 & 1 & 1 & 1 & \cdots & 1 & 1 & 1 \\
1 & 0 & 1 & 1 & \cdots & 1 & 1 & 1 \\
1 & 0 & 0 & 1 & \cdots & 1 & 1 & 1 \\[-5pt]
\vdots & \vdots & \vdots & \vdots & \ddots & \vdots & \vdots & \vdots \\
1 & 0 & 0 & 0 & \cdots & 1 & 1 & 1 \\
1 & 0 & 0 & 0 & \cdots & 0 & 1 & 1 \\
2 & 1 & 1 & 1 & \cdots & 1 & 1 & 2
\end{bsmallmatrix} &
\begin{bsmallmatrix}
0 & 1 & 1 & 1 & \cdots & 1 & 1 & 0 \\
1 & 0 & 1 & 1 & \cdots & 1 & 1 & 0 \\
0 & 1 & 0 & 1 & \cdots & 1 & 1 & 0 \\
0 & 0 & 1 & 0 & \cdots & 1 & 1 & 0 \\[-5pt]
\vdots & \vdots & \vdots & \vdots & \ddots & \vdots & \vdots & \vdots \\
0 & 0 & 0 & 0 & \cdots & 0 & 1 & 0 \\
0 & 0 & 0 & 0 & \cdots & 1 & 0 & 0 \\
1 & 1 & 1 & 1 & \cdots & 1 & 2 & 0
\end{bsmallmatrix} & \geq 2 \\ \hline
\end{array}$}
\]
\end{itemize}
\end{theorem}

To prove Theorem~\ref{thm:classification 2}, we use the following lemma to reduce the problem.

\begin{lemma} \label{lem:reduction}
Let $A$ be a basic $\N$-graded Gorenstein tiled order such that the stable category $\ul{\CM \!}^{\Z} A$ is triangle equivalent to $\per kA_3$. Then $A$ is conjugate to a tiled order $B$ satisfying the following conditions.
\begin{itemize}
\item[a)] All $b$-invariants of $B$ are positive.
\item[b)] All entries of ${\rm v}(B)$ are less than or equal to $2$.
\item[c)] Exactly two entries of ${\rm v}(B)$ equal $2$.
\end{itemize}
In particular, all non-diagonal entries in the same row or column of the two entries $2$ equal $1$.
\end{lemma}

\begin{proof}
Without loss of generality, we may and will assume that all $b$-invariants of $A$ are positive. Since the poset $\V_A$ has only three elements, all $b$-invariants of $A$ must be less than or equal to $3$. If a $b$-invariant $d(\nu i, i)$ of $A$ equals $3$, then either the $\nu i$-th row or the $i$-th column of ${\rm v}(A)$ has an entry greater than or equal to $2$. In both cases $A$ has another $b$-invariant $d(\nu j, j)$ greater than or equal to $2$. If $d(\nu j, j)$ equals $3$, then the element ${\rm v}(e_{\nu j}A(1)_{\geq 0})$ must equal ${\rm v}(e_{\nu i}A(1)_{\geq 0})$ in $\V_A$. It follows that we have $d(\nu i, j)=d(\nu j, i)=3$. By Lemma~\ref{lem:basic}, this is a contradiction. If $d(\nu j, j)$ equals $2$, then the element ${\rm v}(e_{\nu j}A(1)_{\geq 0})$ must equal ${\rm v}(e_{\nu i}A(2)_{\geq 0})$ in $\V_A$. It follows that we have $d(\nu i, j)=3>d(\nu j, j)$. This is a contradiction. Therefore, all $b$-invariants of $A$ must be less than or equal to $2$.

If there are three rows or three columns of ${\rm v}(A)$ which have entries equal to $2$, then three $b$-invariants $d(\nu i, i)$, $d(\nu j, j)$, and $d(\nu k, k)$ of $A$ equal $2$. So all the elements ${\rm v}(e_{\nu i}A(1)_{\geq 0})$, ${\rm v}(e_{\nu j}A(1)_{\geq 0})$, and ${\rm v}(e_{\nu k}A(1)_{\geq 0})$ of $\V_A$ are nonzero. But $\V_A$ has only two nonzero elements. So two of the three elements coincide. Without loss of generality, we may and will assume that we have ${\rm v}(e_{\nu i}A(1)_{\geq 0})={\rm v}(e_{\nu j}A(1)_{\geq 0})$. It follows that we have $d(\nu i, j)=d(\nu j, i)=2$. By Lemma~\ref{lem:basic}, this is a contradiction. Therefore, the entries of ${\rm v}(A)$ which equal to $2$ are concentrated in a submatrix of size $2$. In particular, exactly two $b$-invariants $d(\nu i, i)$ and $d(\nu j, j)$ of $A$ equal $2$.

If the entry $d(\nu i, j)$ equals $2$, by Lemma~\ref{lem:basic}, the entry $d(\nu j, i)$ does not equal $2$. So $d(\nu i, i)$ is the unique entry in the $i$-th column of ${\rm v}(A)$ which equals $2$. It follows that the entry $d(i, i)$ is the unique entry in the $i$-th row of ${\rm v}(A)$ which equals $0$. We replace $A$ by $B=DAD^{-1}$, where $D$ is the diagonal matrix with the $\nu i$-th diagonal entry $x^{-1}$ and the others $1$. In the new matrix ${\rm v}(B)$, the entries which equal to $2$ are still concentrated in the submatrix of size $2$ given by the rows and columns containing the $(\nu j, j)$-entry and the $(\nu^2 i, \nu i)$-entry. Since we have $d(\nu j, \nu i)=d(\nu j, j)-d(\nu i, j)=0$, the $(\nu j, \nu i)$-entry in ${\rm v}(B)$ equals $1$. So only the $(\nu j, j)$-entry and the $(\nu^2 i, \nu i)$-entry in ${\rm v}(B)$ equal $2$. In conclusion, the tiled order $A$ is conjugate to a tiled order $B$ such that exactly two entries of the exponent matrix equal $2$. The last statement is clear.
\end{proof}

\begin{proof}[Proof of Theorem \ref{thm:classification 2}]
The sufficiency in each part follows from part~b) of Theorem~\ref{thm:tilting object}. Let us prove the necessity. Assume that we have $\mathbb{I}=\{1, \ldots, n\}$. We extend the definition of $\mathbb{I}$ to $\Z$ by identifying $i+tn$ with $i$ for all $i\in \mathbb{I}$ and integers $t$.

a) After possibly replacing $A$ with a conjugate $\N$-graded Gorenstein tiled order, we may and will assume that $A$ satisfies all conditions in Lemma~\ref{lem:reduction}. We first consider the case that $A$ is a non-cyclic Gorenstein tiled order. By Lemma~\ref{lem:gluing}, the set $\mathbb{I}$ has exactly two orbits under the Nakayama permutation and each orbit contains an element such that the corresponding $b$-invariant equals to $2$. Moreover, these two orbits have the same cardinality. By Lemma~\ref{lem:cyclic Gorenstein tiled order}, we may and will assume that the exponent matrix ${\rm v}(A)$ is a block matrix such that each block is of size $\frac{n}{2}$ and the diagonal blocks are the exponent matrices associated with cyclic Gorenstein tiled orders. By the proof of part~c) of Theorem~\ref{thm:classification 1}, the diagonal blocks of ${\rm v}(A)$ must be of the form $d_1$ in type $\mathrm{I}$. Since all non-diagonal entries in the same row or column of the two $b$-invariants which equal $2$ are $1$, if we use the equalities $d(i+1, j)+d(j, i)=d(i+1, i)$ iteratively, all entries of the matrix ${\rm v}(A)$ are determined to be the entries in type $\mathrm{I}$.

We now consider the case that $A$ is a cyclic Gorenstein tiled order. If $n$ equals $2$, then the statement is clear. We assume that $n$ is greater than $2$. After possibly conjugating $A$ by a permutation matrix, we may and will assume that we have $\nu i=i+1$ for all $i\in \mathbb{I}$ and $d(1, n)=d(i_0+1, i_0)=2$. If $i_0$ equals $1$, then we have $d(n, 1)=d(2, 1)-d(2, n)=1$ and hence we have
\[
d(n, n-1)=d(n, 1)+d(1, n-1)=2 \: .
\]
But we have $n-1\neq n$, $i_0$, this contradicts to that only the $(1, n)$-entry and the \linebreak $(i_0+1, i_0)$-entry equal to $2$. Thus we have $i_0\neq 1$. Similarly, we have $i_0\neq n-1$. If $2i_0$ is less than $n$, then we have $d(i+1, i)=1$ for all $n-i_0+1\leq i\leq n-1$. Since $d(i_0+1, n)$ equals $1$, we have
\begin{align*}
d(n, i_0) & =d(i_0+1, i_0)-d(i_0+1, n)=1 \: ,\\
d(i_0, n-1) & =d(n, n-1)-d(n, i_0)=0 \: ,\\
 & \ldots \\
d(2, n-i_0+1) & =d(n-i_0+2, n-i_0+1)-d(n-i_0+2, 2)=0 \: , \\
d(n-i_0+1, 1) & =d(2, 1)-d(2, n-i_0+1)=1 \: ,
\end{align*}
respectively. We deduce that we have
\[
d(n-i_0+1, n-i_0)=d(n-i_0+1, 1)+d(1, n-i_0)=2 \: .
\]
But we have $n-i_0\neq n$, $i_0$, this contradicts to that only the $(1, n)$-entry and the \linebreak $(i_0+1, i_0)$-entry equal to $2$. Thus we have $2i_0\geq n$. Then we have $d(i+1, i)=1$ for all $1\leq i\leq n-i_0-1$. Since $d(1, i_0)$ equals $1$, we have
\begin{align*}
d(i_0+1, 1) & =d(i_0+1, i_0)-d(1, i_0)=1 \: ,\\
d(2, i_0+1) & =d(2, 1)-d(i_0+1, 1)=0 \: ,\\
 & \ldots \\
d(n-i_0, n-1) & =d(n-i_0, n-i_0-1)-d(n-1, n-i_0-1)=0 \: , \\
d(n, n-i_0) & =d(n, n-1)-d(n-i_0, n-1)=1 \: ,
\end{align*}
respectively. We deduce that we have $d(n-i_0+1, n-i_0)=d(n-i_0+1, n)+d(n, n-i_0)=2$. But only the $(1, n)$-entry and the $(i_0+1, i_0)$-entry equal to $2$, so we have $n-i_0=i_0$. If we use the equalities $d(i+1, j)+d(j, i)=d(i+1, i)$ iteratively, all entries of the matrix ${\rm v}(A)$ are determined to be the entries in type $\mathrm{II}$.

b) By Lemma~\ref{lem:reduction}, the tiled order $A$ is conjugate to one of the classes given in part~a). Consider the diagonal matrix $D=\diag(x^{l_1}, \ldots, x^{l_n})$ satisfying $l_i=1$ and $l_j=0$ for all $j\neq i$. If $DAD^{-1}$ is $\N$-graded, then all non-diagonal entries in the $i$-th column of ${\rm v}(A)$ must be positive. If $D^{-1}AD$ is $\N$-graded, then all non-diagonal entries in the $i$-th row of ${\rm v}(A)$ must be positive. From this fact, one easily check that types $\mathrm{III}$ and $\mathrm{IV}$ (respectively, $\mathrm{V}$ and $\mathrm{VI}$) are the only ones which are conjugate to type $\mathrm{I}$ (respectively, $\mathrm{II}$). Then the statement follows.
\end{proof}

We end this section with posing the following question.

\begin{question}
For any basic $\N$-graded Gorenstein tiled order $A$ such that $\Ga$ is hereditary, does the existence of $b$-invariants greater than $2$ imply that $\V_A$ is of (not necessarily linear) type $A$?
\end{question}

\subsection{Endomorphism algebras $\Ga_A$ of small global dimension}

Let $m_1$, \ldots, $m_n$ be non-negative integers such that not all of them are zero. Put $d(i, i)=0$, $d(i, j)=\sum_{k=i}^{j-1}m_k$ if $i<j$ and $d(i, j)=\sum_{k=i}^n m_k +\sum_{k=1}^{j-1}m_k$ if $i>j$. By Section~8.2 of \cite{IyamaKimuraUeyama24}, it gives rise to a basic Gorenstein tiled order $A(m_1, \ldots, m_n)$ determined by the exponent matrix $[d(i, j)]$.

\begin{proposition} \label{prop:global dimension 2}
For a basic Gorenstein tiled order $A(m_1, \ldots, m_n)$, the algebra $\Ga$ is of global dimension less than or equal to $2$.
\end{proposition}

\begin{proof}
Given an element ${\rm v}(e_i A(j)_{\geq 0})$ of $\V_A$. For any ${\rm v}(e_{i'} A(j')_{\geq 0}) \in \V_A$ satisfying \linebreak ${\rm v}(e_{i'} A(j')_{\geq 0})<{\rm v}(e_i A(j)_{\geq 0})$ in $\V_A$, if $i'$ equals $i$, then we have
\[
{\rm v}(e_{i'} A(j')_{\geq 0})\leq {\rm v}(e_i A(j+1)_{\geq 0})\: .
\]
If $i'$ does not equal $i$, then we have
\[
{\rm v}(e_{i'} A(j')_{\geq 0})\leq {\rm v}(e_{i-1} A(j+m_{i-1})_{\geq 0})\: .
\]
This implies that ${\rm v}(e_i A(j+1)_{\geq 0})$ and ${\rm v}(e_{i-1} A(j+m_{i-1})_{\geq 0})$ are the only possible maximal elements in the subposet $\{x \in \V_A \mid x<{\rm v}(e_i A(j)_{\geq 0}) \}$. On the other hand, the only possible meet of these two elements in $\V_A$ is ${\rm v}(e_{i-1} A(j+1+m_{i-1})_{\geq 0})$. Then the statement follows from part~b) of Theorem~\ref{thm:tilting object} and Theorem~2.2 of \cite{IyamaMarczinzik22}.
\end{proof}

\begin{lemma} \label{lem:rank 3}
Any basic $\N$-graded Gorenstein tiled order with $|\mathbb{I}|=3$ is of the form $A(m_1, m_2, m_3)$.
\end{lemma}

\begin{proof}
Let $A$ be a basic $\N$-graded Gorenstein tiled order with $|\mathbb{I}|=3$. Then the Nakayama permutation of $A$ has no fixed points and hence $A$ is a cyclic Gorenstein order. After possibly conjugating $A$ by a permutation matrix, we may and will assume that we have
\[
{\rm v}(A)=
\begin{bsmallmatrix}
0 & a & a+b \\
b+c & 0 & b \\
c & c+a & 0
\end{bsmallmatrix}
\]
with non-negative integers $a$, $b$, and $c$. Since $A$ is basic, not all of $a$, $b$, and $c$ are zero.
\end{proof}

\begin{proposition} \label{prop:hereditary of size 3}
Let $A$ be an $\N$-graded Gorenstein tiled order with $|\mathbb{I}|=3$ and $V$ the standard tilting object in $\ul{\CM \!}^{\Z} A$.
\begin{itemize}
\item[a)] The global dimension of $\Ga$ is less than or equal to $2$.
\item[b)] If $A$ is basic, then $\Ga$ is hereditary if and only if $A$ is conjugate to the tiled order determined by the exponent matrix
\[
\begin{bsmallmatrix}
0 & 0 & 0 \\
1 & 0 & 0 \\
1 & 1 & 0
\end{bsmallmatrix} \ko
\begin{bsmallmatrix}
0 & 0 & 0 \\
2 & 0 & 0 \\
2 & 2 & 0
\end{bsmallmatrix} \ko or
\begin{bsmallmatrix}
0 & 1 & 2 \\
2 & 0 & 1 \\
1 & 2 & 0
\end{bsmallmatrix} \: .
\]
\end{itemize}
\end{proposition}

\begin{proof}
a) After possibly replacing $A$ with a graded Morita equivalent $\N$-graded Gorenstein tiled order, we may and will assume that it is basic. Then the statement follows from Lemma~\ref{lem:rank 3} and Proposition~\ref{prop:global dimension 2}.

b) We assume that ${\rm v}(A)$ is of the form in Lemma~\ref{lem:rank 3} satisfying $c\geq a$ and $c\geq b$. Then the integer $c$ must be positive. If $c$ equals $1$, then $a$ and $b$ equal $0$ or $1$. One easily checks that in this case $A$ is conjugate to one of the desired tiled orders. We now consider the case that we have $c\geq 2$. Since $\Ga$ is hereditary, by the implication from ii) to i) in Corollary~\ref{cor:realization of quivers}, for any element $x$ of $\V_A$ except the minimum, the subposet $\{y \in \V_A \mid y<x\}$ has a unique maximal element. Since the elements
$\begin{bmatrix}
c-2 & c+a-2 & 0
\end{bmatrix}$
and
$\begin{bmatrix}
c-1 & 0 & 0
\end{bmatrix}$
are the only possible maximal elements in the subposet
\[
\{x \in \V_A \mid x<
\begin{bmatrix}
c-1 & c+a-1 & 0
\end{bmatrix}
\}
\]
but their first components are different, we must have $c+a-2 \leq 0$. This implies that we have $c=2$ and $a=0$. If $b$ equals $0$ (respectively, $1$), then $A$ is conjugate to the second (respectively, the third) desired tiled orders. If we have $b\geq 2$, then the elements
$\begin{bmatrix}
b & 0 & b-2
\end{bmatrix}$
and
$\begin{bmatrix}
0 & 0 & b-1
\end{bmatrix}$
are the only possible maximal elements in the subposet
\[
\{x \in \V_A \mid x<
\begin{bmatrix}
b+1 & 0 & b-1
\end{bmatrix}
\}
\]
but their third components are different, we must have $b \leq 0$. This contradicts to the hypothesis. This concludes the proof.
\end{proof}



\def\cprime{$'$} \def\cprime{$'$}
\providecommand{\bysame}{\leavevmode\hbox to3em{\hrulefill}\thinspace}
\providecommand{\MR}{\relax\ifhmode\unskip\space\fi MR }
\providecommand{\MRhref}[2]{%
  \href{http://www.ams.org/mathscinet-getitem?mr=#1}{#2}
}
\providecommand{\href}[2]{#2}

\end{document}